\documentclass[final]{siamltex}

\usepackage{graphicx}   % need for figures
\usepackage{multirow}
\usepackage{amsmath}
\usepackage{amsfonts}
\usepackage{latexsym}
\usepackage{verbatim}
\usepackage{amsmath}    % need for subequations

\newtheorem{remark}[theorem]{Remark}

\newcommand{\comments}[1]{}

\newcommand{\beq}{\begin{equation}}
\newcommand{\eeq}{\end{equation}}
\newcommand{\beqa}{\begin{eqnarray}}
\newcommand{\eeqa}{\end{eqnarray}}
\newcommand{\beqas}{\begin{eqnarray*}}
\newcommand{\eeqas}{\end{eqnarray*}}
\newcommand{\bi}{\begin{itemize}}
\newcommand{\ei}{\end{itemize}}
\newcommand{\ba}{\begin{array}}
\newcommand{\ea}{\end{array}}

\newcommand{\nn}{\nonumber}

\def\eqnok#1{(\ref{#1})}

\def\vgap{\vspace*{.1in}}

\def\exp{{\rm exp}}
\def\prob{{\rm Prob}}

\newcommand{\bbe}{\mathbb{E}}
\def\prob{\mathop{\rm Prob}}
\def\Prob{{\hbox{\rm Prob}}}
\newcommand{\bbr}{\mathbb{R}}
\def\w{\omega}

\def\SFO{{\cal SFO}}
\def\SZO{{\cal SZO}}
\def\sigmasco{{\sigma}}

\def\cB{{\cal B}}

\def\cD{{\cal D}}

\def\cE{{\bbr^n}}

\title{
Stochastic First- and Zeroth-order Methods\\
for Nonconvex Stochastic Programming
\thanks{The first author was partially supported by
    NSF Grant CMMI-1000347 and the second author was partially supported
    by NSF grant CMMI-1000347,
    ONR grant N00014-13-1-0036 and
    NSF CAREER Award CMMI-1254446.}
}
\date{June 9, 2012}
\author{
    Saeed Ghadimi
    \thanks{
    Department of Industrial and Systems Engineering,
    University of Florida, Gainesville, FL 32611,
       (email: {\tt sghadimi@ufl.edu}).}
\and
    Guanghui Lan
    \thanks{Department of Industrial and Systems Engineering,
    University of Florida, Gainesville, FL 32611,
       (email: {\tt glan@ise.ufl.edu}). }
}
\begin{document}

\maketitle

\begin{abstract}
%Convexity has become a standard assumption to establish the convergence of
%stochastic approximation (SA) type algorithms in the literature.
In this paper, we introduce a new stochastic approximation (SA) type algorithm, namely
the randomized stochastic gradient (RSG) method, for solving an important class of
nonlinear (possibly nonconvex) stochastic programming (SP) problems.
We establish the complexity of this method for computing an approximate stationary point
of a nonlinear programming problem. We also show that this method possesses
a nearly optimal rate of convergence if the problem is convex. We discuss a variant of the algorithm
which consists of applying a post-optimization phase to evaluate a short list of solutions
generated by several independent runs of the RSG method, and show that such modification allows to improve significantly
the large-deviation properties of the algorithm.
These methods are then specialized for solving a class of simulation-based
optimization problems in which only stochastic zeroth-order information is available.
%We also develop a special version  of the method for solving a class
%of simulation-based optimization problems in which only stochastic zero-order information is available.

\vspace{.1in}


\end{abstract}

%\clearpage

%\setcounter{page}{1}

%\begin{center}
%{\bf Stochastic First- and Zeroth-order Methods\\
%for Nonconvex Stochastic Programming}
%\end{center}
%\vspace{0.1cm}

\setcounter{equation}{0}
\section{Introduction} \label{sec_intro}
%Stochastic programming (SP) has attracted considerable interest during the past decades for its
%applications in a broad spectrum of areas including statistical estimation,
%signal processing and operations research, etc.
%In the classical
%setting, the SP problem is given in the form of $\min_{x \in X} f(x)$, where
%$X$ is a closed convex set and $f: X \to \bbr$ is a strongly
%convex and differentiable function for which only noisy gradient
%information is available.
In 1951, Robbins and Monro in their
seminal work \cite{RobMon51-1} proposed a classical stochastic
approximation (SA) algorithm for solving stochastic programming (SP) problems.
This approach mimics the simplest gradient descent
method by using noisy gradient information in place
of the exact gradients, and possesses the ``asymptotically optimal'' rate of convergence
for solving a class of strongly convex SP problems \cite{chu54,sac58}. However, it is usually
difficult to implement the ``asymptotically optimal'' stepsize
policy, especially in the beginning, so that the algorithms often
perform poorly in practice (e.g., \cite[Section 4.5.3]{Spall03}).
An important improvement of the classical SA was developed by Polyak \cite{pol90} and Polyak and Juditsky
\cite{pol92}, where longer stepsizes were suggested together with the averaging
of the obtained iterates. Their methods were shown to be more robust
with respect to the selection of stepsizes than the classical SA and also exhibit the ``asymptotically
optimal'' rate of convergence for solving strongly convex SP problems.
We refer to \cite{NJLS09-1} for an account of the earlier
history of SA methods.

%there has been a revival of interest in SA methods and their applications
%(e.g., \cite{JNTV05-1,JRT08-1,lns11,NJLS09-1,Nest06-2,Singer07pegasos:primal,Smale05onlinelearning}).
The last few years have seen some significant progress for the development of SA methods for SP.
On one hand, new SA type methods are being introduced to solve SP problems which
are not necessarily strongly convex. On the other hand, these developments, motivated by complexity theory
in convex optimization~\cite{nemyud:83},
concerned the convergence properties of SA methods during a finite number of iterations.
For example, Nemirovski et al. \cite{NJLS09-1} presented a properly modified SA approach,
namely, mirror descent SA for solving general non-smooth convex SP problems.
They demonstrated that the mirror descent SA exhibits an
optimal ${\cal O} ( 1 /\epsilon^2)$ iteration complexity for solving
these problems. This method has been shown in  \cite{lns11,NJLS09-1} to be competitive to
the widely-accepted  sample average approximation approach
(see, e.g., \cite{ksh,sha03}) and even significantly
outperform it for solving a class of convex SP problems.
Similar techniques, based on
subgradient averaging, have been proposed in \cite{JNTV05-1,JRT08-1,Nest06-2}.
While these techniques dealt with non-smooth convex programming problems, Lan~\cite{Lan10-3} presented a unified
optimal method for smooth, non-smooth and stochastic optimization,
which explicitly takes into account the smoothness of the objective function
(see also~\cite{GhaLan12-2a,GhaLan10-1b} for discussions about strong convexity).
However, note that convexity has played an important role in establishing
the convergence of all these SA algorithms. To the best of our knowledge,
none of existing SA algorithms can handle more general SP
problems whose objective function is possibly nonconvex.

This paper focuses on the theoretical development of SA type methods for solving an important class of
nonconvex SP problems. More specifically, we study the classical unconstrained nonlinear programming (NLP) problem
given in the form of
%In this paper,  we intend to develop SA type methods for solving a class of nonconvex SP problems when only noisy information is available. Specifically, we study the classical %nonlinear programming (NLP) of
(e.g.,~\cite{Nest04,NocWri99})
\beq \label{NLP}
f^* := \inf_{x \in \bbr^n} f(x),
\eeq
where $f:\bbr^n \to \bbr$ is a differentiable (not necessarily convex),
bounded from below, and its gradient $\nabla f(\cdot)$ satisfies
\[
\|\nabla f(y) - \nabla f(x)\| \le L\|y-x\|, \ \ \
\forall x, y \in \bbr^n.
\]
However, different from the standard NLP, we assume throughout the paper that
we only have access to noisy function values or gradients about the objective function $f$
in \eqnok{NLP}. In particular, in the basic setting, we assume that problem \eqnok{NLP} is to be solved by iterative algorithms
which acquire the gradients of $f$ via subsequent calls to a stochastic first-order
oracle ($\SFO$). At iteration $k$ of the algorithm, $x_k$ being the input,
the $\SFO$ outputs a {\sl stochastic gradient} $G(x_k, \xi_k)$, where
$\xi_k$, $k \ge 1$, are random variables whose distributions $P_k$ are supported on $\Xi_k \subseteq \bbr^d$.
The following assumptions are made for the Borel functions $G(x_k, \xi_k)$.

\vgap

{\bf A1:} For any $k \ge 1$, we have
\beqa
&\mbox{a)}& \, \, \bbe [G(x_k, \xi_k)] = \nabla f(x_k), \label{ass1.a} \\
&\mbox{b)} & \, \, \bbe \left[ \|G(x_k, \xi_k) - \nabla f(x_k)\|^2 \right] \le \sigma^2, \label{ass1.b}
\eeqa
for some parameter $\sigma \ge 0$.
Observe that, by \eqnok{ass1.a}, $G(x_k, \xi_k)$ is an unbiased estimator of $\nabla f(x_k)$
and, by \eqnok{ass1.b}, the variance of the random variable $\|G(x_k, \xi_k) - \nabla f(x_k)\|$ is bounded.
It is worth noting that in the standard setting for SP,
the random vectors $\xi_k$, $k = 1,2, \ldots$, are independent of each other 
(and also of $x_k$) (see, e.g., \cite{nemyud:83,NJLS09-1}). 
Our assumption here is slightly weaker since we do not need to assume  $\xi_k$, $k = 1,2, \ldots$, to be independent.

%For the sake of future reference, we call problem \eqnok{NLP} equipped with the $\SFO$ defined above as an nonconvex SP problems
%with stochastic first-order information.
%Indeed, we are assuming that at each point $x$, the stochastic gradient is an unbiased estimator of the true gradient of $f(x)$ and its variance is bounded by $\sigma^2$. It %should be mentioned that in Assumption A1, we don't require ${\xi_k}_{k \ge 1}$ to be independently and identically distributed (i.i.d) random variables, which is a common %assumption for the stochastic oracles. Note also that,

Our study on the aforementioned SP problems
has been motivated by a few interesting applications which are briefly
outlined as follows.
%To motivate our study, let us first point out a few applications which can
%be formulated in the above form.
\begin{itemize}
\item In many {\sl machine learning} problems,
we intend to minimize a regularized loss function $f(\cdot)$ given by
\beq \label{case1}
 f(x) = \int_{\Xi} L(x,\xi) dP(\xi) + r(x),
\eeq
where either the loss function $L(x,\xi)$ or the regularization
$r(x)$ is nonconvex (see, e.g.,~\cite{Mairal09,MasBaxBarFre99}).
\item Another important class of problems originate from the
so-called {\sl endogenous uncertainty} in SP.
More specifically, the objective functions for these SP problems are given
in the form of
\beq \label{case2}
f(x) = \int_{\Xi(x)} F(x,\xi) dP_x(\xi),
\eeq
where the support $\Xi(x)$ and the distribution function $P_x$ of
the random vector $\xi$ depend on $x$. The function $f$
in \eqnok{case2} is usually nonconvex
even if $F(x, \xi)$ is convex with respect to $x$.
For example, if the support $\Xi$ does not depend on $x$,
it is often possible to represent $dP_x=H(x)dP$ for some fixed distribution $P$.
Typically this transformation results in a nonconvex integrand function.
Other techniques have also been developed to compute unbiased estimators for the gradient of $f(\cdot)$
in \eqnok{case2} (see, e.g., \cite{Fu06a,Glasserman91,LE90-1,RubSha93}).
%A few different techniques
%If the support $\Xi$ does not depend on $x$, then often it can be transformed
%into expression with constant (independent of $x$) distribution. There are basically two approaches.
%One, it is often possible to represent $dP_x=H(x)dP$ for some fixed distribution $P$
%for example for normal distribution where the mean and covariance matrix are functions of $x$).
%Another approach is the Likelihood Ratio method (which is very unstable). Typically this results in nonconvex
%integrand function. If the domain $\Xi(x)$ is a function of $x$, then it is really difficult
%to estimate the corresponding gradients. Although such formulas do exist, they are not practical.
\item
Finally, in {\sl simulation-based optimization},
the objective function is given by $f(x) = \bbe_\xi[F(x,\xi)]$, where
$F(\cdot, \xi)$ is not given explicitly, but through a black-box simulation
procedure (e.g., \cite{Andr98-1,Fu02-1}). Therefore,
we do not know if the function $f$ is convex or not.
Moreover, in these cases, we usually only have access to stochastic
zeroth-order information about the function values of $f(\cdot)$ rather than its gradients.
\end{itemize}

The complexity of the gradient descent method for solving problem~\eqnok{NLP} has been well-understood
under the deterministic setting (i.e., $\sigma = 0$ in \eqnok{ass1.b}).
In particular, Nesterov~\cite{Nest04} shows that after running the
method for at most $N = {\cal O}(1/\epsilon)$ steps, we have
$\min_{k=1, \ldots, N} \|\nabla f(x_k)\|^2 \le \epsilon$ (see Gratton et al.~\cite{GraSarToi08}
for a similar bound for the trust-region methods).
Cartis et al.~\cite{CarGouToi10-1} show that this bound is actually
tight for the gradient descent method. Note, however, that the analysis in \cite{Nest04}
is not applicable to the stochastic setting (i.e., $\sigma > 0$ in \eqnok{ass1.b}). Moreover, even if we have
$\min_{k=1, \ldots, N} \|\nabla f(x_k)\|^2 \le \epsilon$, to
find the best solution from $\{x_1, \ldots, x_N\}$ is still difficult
since $\|\nabla f(x_k)\|$ is not known exactly.
%To solve problem \eqnok{NLP} under the stochastic setting,
Our major contributions in this paper are summarized as follows. Firstly,
to solve the aforementioned nonconvex SP problem, we present a randomized stochastic gradient (RSG) method
by introducing the following modifications to the classical SA.
Instead of taking average of the iterates as in the mirror descent SA
for convex SP, we randomly select a solution $\bar x$ from
$\{x_1, \ldots, x_N\}$ according to a certain probability distribution as the output. We show that such a solution satisfies
$\bbe[\|\nabla f(\bar x)\|^2] \le \epsilon$ after running the method for at most $N = {\cal O}(1/\epsilon^2)$
iterations~\footnote{It should not be too surprising to see that
the complexity for the stochastic case is much worse than that for the deterministic case.
For example, in the convex case, it is known~\cite{Nest04,Lan10-3} that the complexity for finding an solution $\bar x$ satisfying
$f(\bar x) - f^* \le \epsilon$ will be substantially increased from ${\cal O}(1/\sqrt{\epsilon})$ to
${\cal O}(1/\epsilon^2)$ as one moves from the deterministic to stochastic setting.}.
Moreover, if $f(\cdot)$ is convex, we show that the relation $\bbe[f(\bar x) - f^*] \le \epsilon$ always holds.
%Therefore, the RSG method possesses an optimal iteration complexity, in terms of its dependence on $\epsilon$,
%for solving convex SP problems.
We demonstrate that such a complexity result is nearly optimal for solving convex SP problems (see the discussions after
Corollary~\ref{noncvx_smth_cor}).
%It is also worth noting that, while the convergence
%of most existing SA requires the random variables $\xi$ to be independent of $x$,
%the developed RSG method does not rely on such a requirement as
%long as Assumption A1 is satisfied.

Secondly, in order to improve the large deviation properties and hence the
reliability of the RSG method, we present a two-phase randomized stochastic gradient ($2$-RSG)
method by introducing a post-optimization phase
to evaluate a short list of solutions
generated by several independent runs of the RSG method.
We show that the complexity of the $2$-RSG method for
computing an {\sl $(\epsilon, \Lambda)$-solution} of problem \eqnok{NLP},
i.e.,  a point $\bar x$ such that
$\Prob\{\|\nabla f(\bar x)\|^2 \le \epsilon\} \ge 1-\Lambda$
for some $\epsilon > 0$ and $\Lambda \in (0,1)$,
can be bounded by
\[
{\cal O} \left\{
 \frac{\log(1/\Lambda) \sigma^2}{\epsilon}\left[\frac{1}{\epsilon}
 +  \frac{\log(1/\Lambda)}{\Lambda}\right]
\right\}.
\]
We further show that, under certain light-tail assumption about the $\SFO$,
the above complexity bound can be reduced to
\[
{\cal O} \left\{
 \frac{\log(1/\Lambda) \sigma^2}{\epsilon}\left(\frac{1}{\epsilon}
 +  \log\frac{1}{\Lambda}\right)
\right\}.
\]

Thirdly, we specialize the RSG method for the case where only stochastic
zeroth-order information is available. There exists a somewhat long history
for the development of zeroth-order (or derivative-free) methods in nonlinear programming
(see the monograph by Conn et al.~\cite{ConSchVic09} and references therein).
However, only few complexity results are available
for these types of methods, mostly for convex programming (e.g.,~\cite{nemyud:83,Nest11-1}) and
deterministic nonconvex programming problems (e.g.,~\cite{CarGouToi12,GarVic12,Nest11-1,Vinc12}).
The stochastic zeroth-order methods studied in this paper are directly motivated by
a recent important work due to Nesterov~\cite{Nest11-1}. More specifically, Nesterov proved in
\cite{Nest11-1} some tight bounds for approximating first-order
information by zeroth-order information using the Gaussian smoothing
technique (see Theorem~\ref{smth_approx}). Based on this technique,
he presented a series of new complexity results for zeroth-order methods.
For example, he established the 
${\cal O} \left(n /\epsilon \right)$ complexity, in terms of
$\bbe[f(\bar x) - f^*] \le \epsilon$, for a zeroth-order method applied to smooth convex programming problems (see in p.19 of \cite{Nest11-1})
along with some possible acceleration schemes. Here the expectation is taken with respect to the Gaussian random variables used in
the algorithms.
He had also proved the ${\cal O} (n/\epsilon)$ complexity, in terms of $\bbe[\|\nabla f (\bar x)\|^2] \le \epsilon$, 
for solving smooth nonconvex problems (see p.24 of \cite{Nest11-1}).
While these bounds were obtained for solving deterministic optimization
problems, Nesterov established the ${\cal O} (n^2/\epsilon^2)$
complexity, in terms of $\bbe[f(\bar x) - f^*] \le \epsilon$, for solving general nonsmooth convex
SP problems (see p.17 of \cite{Nest11-1}).

By incorporating the Gaussian smoothing technique~\cite{Nest11-1} into the RSG method,
we present a randomized stochastic gradient free (RSGF) method
for solving a class of simulation-based optimization problems and
demonstrate that its iteration complexity for finding
the aforementioned $\epsilon$-solution (i.e., $\bbe[\|\nabla f(\bar x)\|^2] \le \epsilon$)
can be bounded by ${\cal O}(n/\epsilon^2)$. To the best of our knowledge,
this appears to be the first complexity result for nonconvex stochastic zeroth-order methods in the literature.
Moreover, the same RSGF algorithm possesses an ${\cal O}(n/\epsilon^2)$
complexity bound, in terms of $\bbe[f(\bar x) - f^*] \le \epsilon$,
for solving smooth convex SP problems. It is interesting to observe that
this bound has a much weaker dependence on $n$ than 
the one previously established by Nesterov for solving general nonsmooth convex SP problems (see p.17 of \cite{Nest11-1}). 
% since the generated solution $\bar x$ always satisfies $\bbe[f(\bar x) - f^*] \le \epsilon$ whenever $f$ is convex,
%we have shown that the complexity result
%in \cite{Nest11-1} for general nonsmooth convex SP (see p.17 of \cite{Nest11-1})
%can be strengthened to ${\cal O} (n/\epsilon^2)$ for solving smooth convex SP problems.
Such an improvement is obtained by explicitly making use of the smoothness properties of
the objective function and carefully choosing the stepsizes and smoothing
parameter used in the RSGF method.

This paper is organized as follows. We introduce two stochastic first-order methods,
i.e., the RSG and $2$-RSG methods,
for nonconvex SP, and establish their convergence properties in Section 2.
We then specialize these methods for solving a class of simulation-based optimization
problems in Section 3. Some brief concluding remarks are also presented in Section 4.

\subsection{Notation and terminology} \label{notation-sco}
As stated in \cite{Nest04}, we say that $f\in {\cal C}_L^{1,1}(\cE)$ if it is differentiable and
\[
\|\nabla f(y) - \nabla f(x)\| \le L\|y-x\|, \ \ \
\forall x, y \in \cE.
\]
Clearly, we have
\beq \label{smooth}
|f(y) - f(x) - \langle \nabla f(x), y - x \rangle | \le \frac{L}{2} \|y - x\|^2, \ \ \
\forall x, y  \in \cE.
\eeq
If, in addition, $f(\cdot)$ is convex, then
\beq \label{smooth_lb1}
f(y) - f(x) - \langle \nabla f(x), y - x \rangle \ge \frac {1}{2L} \|\nabla f(y) - \nabla f(x)\|^2,
\eeq
and
\beq \label{smooth_lb}
\langle \nabla f(y) - \nabla f(x), y - x \rangle \ge \frac {1}{L} \|\nabla f(y) - \nabla f(x)\|^2, \ \ \
\forall x, y \in \cE.
\eeq

\setcounter{equation}{0}
\section{Stochastic first-order methods} \label{sec_first}
Our goal in this section is to present and analyze a new class of SA algorithms
for solving general smooth nonlinear (possibly nonconvex) SP problems.
More specifically, we present the RSG method and
establish its convergence properties in Subsection~\ref{sec_RSG},
and then introduce the $2$-RSG method
which can significantly improve the large-deviation properties of the RSG method
in Subsection~\ref{sec_2RSG}.

We assume throughout this section that Assumption A1 holds.
In some cases, Assumption A1 is augmented by the following
``light-tail'' assumption.

\vgap

{\bf A2:} For any $x \in \bbr^n$ and $k \ge 1$, we have
\beq
\bbe\left[\exp \{\| G(x, \xi_k) - g(x) \|^2/\sigmasco^2 \}
\right] \leq \exp\{1\}. \label{ass2.a}
\eeq
It can be easily seen that Assumption A2 implies Assumption A1.b) by Jensen's inequality.

\subsection{The randomized stochastic gradient method} \label{sec_RSG}
%In this section, we consider the classic nonlinear programming (not necessarily convex)
%problem of
%We assume that problem \eqnok{NLP} is to be solved
%by iterative algorithms which acquire the information of $f$ via subsequent calls to a {\sl stochastic
%first-order oracle ($\SFO$)}. Specifically, at iteration $k$ of the algorithm, $x_k \in \cE$ being the input,
%the $\SFO$ outputs a vector $G(x_k, \xi_k)$, where
%$\xi_k$, $k \ge 1$, are random variables whose distribution $P_k$ are supported on $\Xi_k \subseteq \bbr^d$.
%The following assumptions are made for the Borel functions $G(x_k, \xi_k)$.

%\vgap

%{\bf A1:} For any $k \ge 1$, we have
%\beqa
%&\mbox{a)}& \, \, \bbe [G(x_k, \xi_k)] = g(x_k) \equiv \nabla f(x_k) \label{ass1.a} \\
%&\mbox{b)} & \, \, \bbe \left[ \|G(x_k, \xi_k) - g(x_k)\|_*^2 \right] \le \sigma^2, \label{ass1.b}
%\eeqa

\vgap

%Consider problem \eqnok{main_problem} where the objective function $f$ has Lipschitz
%continuous gradient with constant L and it is not necessarily convex.
The convergence of existing SA methods requires
$f(\cdot)$ to be convex~\cite{NJLS09-1,lns11,Lan10-3,GhaLan12-2a,GhaLan10-1b}.
Moreover, in order to guarantee the convexity of $f(\cdot)$,
one often need to assume that the random variables $\xi_k$, $k\ge 1$,
to be independent of the search sequence $\{x_k\}$. Below we present a new SA-type algorithm that
can deal with both convex and nonconvex SP problems, and allow random noises
to be dependent on the search sequence. This algorithm is obtained by
incorporating a certain randomization scheme into the classical SA method.

\vgap

\noindent {\bf A randomized stochastic gradient (RSG) method}
\begin{itemize}
\item [] {\bf Input:} Initial point $x_1$,
iteration limit $N$, stepsizes $\{\gamma_k\}_{k \ge 1}$
and probability mass function $P_{R}(\cdot)$ supported on $\{1, \ldots, N\}$.
\item[] {\bf Step} $0$. Let $R$ be a random variable with probability mass function $P_{R}$.
\item [] {\bf Step} $k = 1, \ldots, R$.
Call the stochastic first-order oracle for computing $G(x_k, \xi_k)$ and
set
\beq \label{update_SA}
x_{k+1} = x_k - \gamma_k G(x_k, \xi_k).
\eeq
\item[]  {\bf Output} $x_R$.
\end{itemize}

A few remarks about the above RSG method are in order.
Firstly, in comparison with the classical SA, we have used
a random iteration count, $R$, to terminate the execution of the RSG algorithm.
Equivalently, one can view such a randomization scheme from a slightly
different perspective described as follows. Instead of terminating
the algorithm at the $R$-th step, one can also run the RSG algorithm for $N$
iterations but randomly choose a search point $x_R$ (according to $P_R$) from its trajectory
as the output of the algorithm. Clearly, using the latter scheme, we just
need to run the algorithm for the first $R$ iterations and the remaining
$N - R$ iterations are surpluses. Note however, that the primary goal to
introduce the random iteration count $R$ is to derive new complexity results for
nonconvex SP, rather than save the computational efforts in the last $N-R$ iterations
of the algorithm. Indeed, if $R$ is uniformly distributed, %(see Corollary~\ref{noncvx_smth_cor}),
the computational gain from such a randomization scheme is simply a factor of
$2$. Secondly, the RSG algorithm described above
is conceptual only because we have not specified the selection of the
stepsizes $\{\gamma_k\}$ and the probability mass function $P_R$ yet.
We will address this issue after establishing some basic convergence
properties of the RSG method.
%Finally,
%it is worth noting that some preliminary ideas about randomizing the trajectory of SA methods
%had appeared in~\cite{Singer07pegasos:primal}
%for solving a special class of strongly convex SP problems.
%However, our development significantly differs from ~\cite{Singer07pegasos:primal}
%on the problem setup, the algorithm, the convergence analysis and the obtained complexity results.
%Moreover, we would like to point out that, when $f(\cdot)$ is strongly convex, one can actually
%obtain better performance guarantees without employing
%such randomization techniques (see, e.g., ~\cite{GhaLan10-1a,GhaLan10-1b}), and hence
%the power of randomizing the trajectory seems to exist in the
%solutions of nonconvex SP problems.

\vgap

The following result describes some convergence properties of the RSG method.

\begin{theorem} \label{noncvx_smth_theorem}
%Let $\{x_k\}_{k \ge 0}$ be the sequence computed according to
%\eqnok{update_SA} by the RSG method.
Suppose that the stepsizes $\{\gamma_k\}$ and
the probability mass function $P_R(\cdot)$
in the RSG method are chosen such that $\gamma_k < 2/ L$ and
\beq \label{prob_fun}
P_R(k) := \prob\{R=k\} = \frac{2 \gamma_k- L\gamma_k^2}
{{\sum_{k=1}^N (2 \gamma_k- L \gamma_k^2)}}, \ \ k= 1,...,N.
\eeq
Then, under Assumption A1,
\begin{itemize}
\item [a)]for any $N \ge 1$, we have
\beq \label{noncvx_smth_converg}
\frac{1}{L} \bbe[\|\nabla f(x_R)\|^2]
\le \frac{D_f^2 + \sigma^2 {\sum_{k=1}^N \gamma_k^2}}{{\sum_{k=1}^N (2\gamma_k-L\gamma_k^2)}},
\eeq
where the expectation is taken with respect to $R$ and $\xi_{[N]} := (\xi_1,...,\xi_N)$,
\beq \label{def_Df}
D_f := \left[\frac{2\left(f(x_1) - f^*\right)}{L}\right]^\frac{1}{2},
\eeq
and $f^*$ denotes the optimal value of problem \eqnok{NLP};
\item [b)] if, in addition, problem \eqnok{NLP} is convex with an optimal solution $x^*$,
%If the stepsizes $\{\gamma_t\}$ and
%the probability mass function ${\cal P}$ for $\zeta$ are chosen such that
%\beq \label{cvx_smth_step_cond}
%\gamma_k \le \frac{1}{L} \ \ \
%\mbox{and} \ \ \
%P(\zeta=k)= \frac{\gamma_k- L\gamma_k^2}
%{{\sum_{k=1}^N (\gamma_k- L \gamma_k^2)}}, \ \ k= 1,...,N,
%\eeq
then, for any $N \ge 1$,
\beq \label{cvx_smth_converg}
\bbe[f(x_R) - f^*]
\le \frac{D_X^2
+ \sigma^2 {\sum_{k=1}^N \gamma_k^2}}{{\sum_{k=1}^N (2\gamma_k-L\gamma_k^2)}},
\eeq
where the expectation is taken with respect to $R$ and $\xi_{[N]}$,
and
\beq \label{def_Dx}
D_X := \|x_1-x^*\|.
\eeq
\end{itemize}
\end{theorem}

\vgap
\begin{proof}
Display $\delta_k \equiv G(x_k, \xi_k) - \nabla f(x_k)$, $k \ge 1$.
We first show part a). Using the assumption that $f \in {\cal C}_L^{1,1}(\bbr^n)$, \eqnok{smooth} and
\eqnok{update_SA}, we have, for any $k = 1, \ldots, N$,
\beqa
f(x_{k+1}) &\le& f(x_k) + \langle \nabla f(x_k), x_{k+1}-x_k \rangle + \frac{L}{2} \gamma_k^2 \|G(x_k, \xi_k)\|^2 \nn \\
&=& f(x_k) -\gamma_k \langle \nabla f(x_k), G(x_k, \xi_k) \rangle + \frac{L}{2} \gamma_k^2 \|G(x_k, \xi_k)\|^2 \nn \\
%&=& f(x_k) -\gamma_k \|\nabla f(x_k)\|^2 - \gamma_k \langle \nabla f(x_k), \delta_k \rangle + \frac{L}{2} \gamma_k^2 \|G(x_k, \xi_k)\|^2\\
&=& f(x_k) -\gamma_k \|\nabla f(x_k)\|^2 - \gamma_k \langle \nabla f(x_k), \delta_k \rangle + \frac{L}{2} \gamma_k^2
\left[\|\nabla f(x_k)\|^2 + 2 \langle \nabla f(x_k), \delta_k \rangle + \|\delta_k\|^2  \right] \nn \\
&=& f(x_k) - \left(\gamma_k - \frac{L}{2} \gamma_k^2\right) \|\nabla f(x_k)\|^2
- \left(\gamma_k - L \gamma_k^2 \right) \langle \nabla f(x_k), \delta_k \rangle + \frac{L}{2} \gamma_k^2 \|\delta_k\|^2. \label{noncvx_smth_cnvg1}
\eeqa
Summing up the above inequalities and re-arranging the terms, we obtain
\beqa
\sum_{k=1}^N \left(\gamma_k - \frac{L}{2} \gamma_k^2\right) \|\nabla f(x_k)\|^2
&\le& f(x_1) - f(x_{N+1}) - \sum_{k=1}^N \left(\gamma_k - L \gamma_k^2 \right) \langle \nabla f(x_k), \delta_k \rangle +
\frac{L}{2} \sum_{k=1}^N  \gamma_k^2 \|\delta_k\|^2 \nn \\
&\le& f(x_1) - f^* - \sum_{k=1}^N \left(\gamma_k - L \gamma_k^2 \right) \langle \nabla f(x_k), \delta_k \rangle +
\frac{L}{2} \sum_{k=1}^N  \gamma_k^2 \|\delta_k\|^2,\label{main_recursion}
\eeqa
where the last inequality follows from the fact that $f(x_{N+1}) \ge f^*$.
Note that the search point $x_k$ is a function of
the history $\xi_{[k-1]}$ of the
generated random process and hence is random.
Taking expectations (with respect to $\xi_{[N]}$)
on both sides of \eqnok{main_recursion}
and noting that under Assumption A1, $\bbe[\|\delta_k\|^2]
\le \sigma^2$,
 and
\beq \label{fh_seq0}
\bbe[\langle \nabla f(x_k), \delta_k \rangle |\xi_{[k-1]}] = 0,
\eeq
we obtain
\beq \label{main_recursion2}
\sum_{k=1}^N \left(\gamma_k - \frac{L}{2} \gamma_k^2\right) \bbe_{\xi_{[N]}}\|\nabla f(x_k)\|^2
\le f(x_1) - f^* +
\frac{L\sigma^2}{2} \sum_{k=1}^N  \gamma_k^2
\eeq
Dividing both sides of the above inequality by
$L \sum_{k=1}^N \left(\gamma_k - L \gamma_k^2 / 2\right)$ and
noting that
\[
\bbe[\|\nabla f(x_R)\|^2] = \bbe_{R, \xi_{[N]}}[\|\nabla f(x_R)\|^2] =
\frac{\sum_{k=1}^N \left(2\gamma_k - L \gamma_k^2\right) \bbe_{\xi_{[N]}}\|\nabla f(x_k)\|^2}
{\sum_{k=1}^N \left(2\gamma_k - L \gamma_k^2\right)},
\]
we conclude
\[
\frac{1}{L} \bbe[\|\nabla f(x_R)\|^2] \le \frac{1}{{\sum_{k=1}^N (2\gamma_k-L\gamma_k^2)}} \left[
\frac{2\left(f(x_1) - f^*\right)}{L} + \sigma^2 {\sum_{k=1}^N \gamma_k^2}\right],
\]
which, in view of \eqnok{def_Df}, clearly implies \eqnok{noncvx_smth_converg}.

We now show that part b) holds. Display $\w_k \equiv \|x_k - x^*\|$.
First observe that, for any $k = 1, \ldots, N$,
\beqas
\lefteqn{\w_{k+1}^2 = \|x_k - \gamma_k G(x_k, \xi_k) - x^*\|^2 }\\
&=& \w_k^2 - 2 \gamma_k \langle G(x_k, \xi_k), x_k - x^*\rangle
+ \gamma_k^2 \|G(x_k, \xi_k)\|^2 \\
&=& \w_k^2 - 2 \gamma_k \langle \nabla f(x_k) + \delta_k, x_k - x^*\rangle
+ \gamma_k^2 \left(\|\nabla f(x_k)\|^2 + 2\langle \nabla f(x_k), \delta_k \rangle
+ \|\delta_k\|^2 \right).
\eeqas
Moreover, in view of \eqnok{smooth_lb} and the fact that $\nabla f(x^*) = 0$, we have
\beq \label{basic_convex_rel}
\frac{1}{L} \|\nabla f(x_k)\|^2 \le  \langle \nabla f(x_k), x_k - x^*\rangle.
\eeq
Combining the above two relations, we obtain, for any $k = 1, \ldots, N$,
\beqas
\w_{k+1}^2 &\le& \w_k^2
- (2\gamma_k - L \gamma_k^2) \langle \nabla f(x_k), x_k - x^* \rangle
- 2\gamma_k \langle x_k - \gamma_k \nabla f(x_k) - x^*, \delta_k \rangle
+ \gamma_k^2 \|\delta_k\|^2\\
&\le& \w_k^2
- (2\gamma_k - L \gamma_k^2) [ f(x_k) - f^*]
- 2\gamma_k \langle x_k - \gamma_k \nabla f(x_k) - x^*, \delta_k \rangle
+ \gamma_k^2 \|\delta_k\|^2,
\eeqas
where the last inequality follows from the convexity of $f(\cdot)$
and the fact that $\gamma_k \le 2/L$.
Summing up the above inequalities and re-arranging the terms, we
have
\beqas
\sum_{k=1}^N(2\gamma_k - L \gamma_k^2) [ f(x_k) - f^*]
&\le& \w_1^2 - \w_{N+1}^2 -2 \sum_{k=1}^N \gamma_k \langle x_k -
\gamma_k \nabla f(x_k) - x^*, \delta_k \rangle + \sum_{k=1}^N \gamma_k^2 \|\delta_k\|^2\\
&\le& D_X^2 - 2 \sum_{k=1}^N \gamma_k \langle x_k -
\gamma_k \nabla f(x_k) - x^*, \delta_k \rangle + \sum_{k=1}^N \gamma_k^2 \|\delta_k\|^2,
\eeqas
where the last inequality follows from \eqnok{def_Dx} and the fact that
$\w_{N+1} \ge 0$. The rest of the proof is similar to that of part a) and hence the details
are skipped.

%Taking expectations (conditional on $\xi_{[k-1]}$) of both sides of \eqnok{noncvx_smth_cnvg1} and under assumption A1 we have,
%\beqa
%\bbe[f(x_{t+1})] &\le& \bbe[f(x_t)] -\gamma_t \|\nabla f(x_t)\|^2 + \frac{L}{2} \gamma_t^2 (\|\nabla f(x_t)\|^2 + \sigma^2) \nn %\\
%&=& \bbe[f(x_t)] -(\gamma_t-\frac{L}{2} \gamma_t^2) \|\nabla f(x_t)\|^2 + \frac{L}{2} \gamma_t^2 \sigma^2 \nn \\
%&\Rightarrow& (\gamma_t-\frac{L}{2} \gamma_t^2) \|\nabla f(x_t)\|^2 \le \bbe[f(x_t)]-\bbe[f(x_{t+1})]+ \frac{L}{2} \gamma_t^2 %\sigma^2
%\label{noncvx_smth_cnvg2}
%\eeqa
%Summing up the above inequality, we obtain

%\beqa
%\sum\limits_{t=0}^N (\gamma_t-\frac{L}{2}\gamma_t^2)\|\nabla f(x_t)\|^2 \le f(x_0)-\bbe[f(x_{N})]+ {\sum\limits_{t=0}^N %\frac{L}{2}\gamma_t^2 \sigma^2}
%\label{noncvx_smth_cnvg3}
%\eeqa
%Assuming that $f^* \le f(x), \ \ \forall x \in \bbr^d$, we have
%\beqa
%\sum\limits_{t=0}^N (\gamma_t-\frac{L}{2}\gamma_t^2)\|\nabla f(x_t)\|^2 \le f(x_0)-f^*+ {\sum\limits_{t=0}^N %\frac{L}{2}\gamma_t^2 \sigma^2}
%\label{noncvx_smth_cnvg4}
%\eeqa
%Dividing both sides of the \eqnok{noncvx_smth_cnvg4} by $\sum\limits_{t=0}^N (\gamma_t-\frac{L}{2}\gamma_t^2)$, we obtain %\eqnok{noncvx_smth_converg}.
\end{proof}

\vgap

We now describe a possible strategy for the selection of the stepsizes $\{\gamma_k\}$
in the RSG method. For the sake of simplicity, let us assume that a constant
stepsize policy is used, i.e., $\gamma_k = \gamma$, $k = 1, \ldots, N$, for some $\gamma \in (0, 2/L)$.
Note that the assumption of constant stepsizes does not hurt the efficiency estimate of the
RSG method. The following corollary of Theorem~\ref{noncvx_smth_theorem} is obtained
by appropriately choosing the parameter $\gamma$.

\begin{corollary} \label{noncvx_smth_cor}
Suppose that the stepsizes $\{\gamma_k\}$ are set to
\beq \label{constant_step}
\gamma_k = \min \left\{ \frac{1}{L}, \frac{\tilde D}{\sigma \sqrt{N}}\right\}, k = 1, \ldots, N,
\eeq
for some $\tilde D > 0$. Also assume that the probability mass function $P_R(\cdot)$
is set to \eqnok{prob_fun}. Then, under Assumption A1, we have
\beq \label{nocvx_smooth}
\frac{1}{L} \bbe[\|\nabla f(x_R)\|^2] \le \cB_N:=
\frac{L D_f^2}{N} + \left(\tilde D + \frac{D_f^2}{\tilde D} \right) \frac{\sigma}{\sqrt{N}},
\eeq
where $D_f$ is defined in \eqnok{def_Df}.
If, in addition, problem \eqnok{NLP} is convex with an optimal solution $x^*$,
then
\beq \label{cvx_smooth}
\bbe[f(x_R) - f^*] \le \frac{L D_X^2}{N} + \left(\tilde D + \frac{D_X^2}{\tilde D} \right)
\frac{\sigma}{\sqrt{N}},
\eeq
where $D_X$ is defined in \eqnok{def_Dx}.
\end{corollary}

\begin{proof}
Noting that by \eqnok{constant_step}, we have
\beqas
\frac{D_f^2 + \sigma^2 {\sum_{k=1}^N \gamma_k^2}}{{\sum_{k=1}^N (2\gamma_k-L\gamma_k^2)}}
&=& \frac{D_f^2 + N \sigma^2 \gamma_1^2}{N \gamma_1 (2-L\gamma_1)}
\le \frac{D_f^2 + N \sigma^2 \gamma_1^2}{N \gamma_1} = \frac{D_f^2}{N \gamma_1} + \sigma^2 \gamma_1\\
&\le& \frac{D_f^2}{N} \max\left\{L, \frac{\sigma \sqrt{N}}{\tilde D} \right\}
+ \sigma^2 \frac{\tilde D}{\sigma \sqrt{N}} \\
&\le& \frac{L D_f^2}{N} + \left(\tilde D + \frac{D_f^2}{\tilde D} \right) \frac{\sigma}{\sqrt{N}},
\eeqas
which together with \eqnok{noncvx_smth_converg} then imply \eqnok{nocvx_smooth}.
Relation \eqnok{cvx_smooth} follows similarly from the above inequality (with $D_f$ replaced by $D_X$)
and \eqnok{cvx_smth_converg}.
\end{proof}

\vgap

%Below we present an increasing stepsize policy.

%\begin{corollary}
%Suppose that the stepsizes $\{\gamma_k\}$ are set to
%\beq \label{constant_step}
%\gamma_k = \min \left\{ \frac{1}{L}, \frac{\tilde D \sqrt{k}}{\sigma N}\right\}, k = 1, \ldots, N,
%\eeq
%for some $\tilde D > 0$. Also assume that the probability mass function $P_R(\cdot)$
%is set to \eqnok{prob_fun}. Then, under Assumption A1, we have
%\beq \label{nocvx_smooth}
%\frac{1}{L} \bbe[\|\nabla f(x_R)\|^2] \le \cB_N:=
%\frac{L D_f^2}{N} + \left(\tilde D + \frac{D_f^2}{\tilde D} \right) \frac{\sigma}{\sqrt{N}}.
%\eeq
%If, in addition, problem \eqnok{NLP} is convex with an optimal solution $x^*$,
%then
%\beq \label{cvx_smooth}
%\bbe[f(x_R) - f^*] \le \frac{L D_X^2}{N} + \left(\tilde D + \frac{D_X^2}{\tilde D} \right)
%\frac{\sigma}{\sqrt{N}}.
%\eeq
%\end{corollary}

We now add a few remarks about the results obtained in Theorem \ref{noncvx_smth_theorem}
and Corollary~\ref{noncvx_smth_cor}.
%we would like explain the role of the randomization scheme in the RSG method.
Firstly, as can be seen from \eqnok{main_recursion2}, instead of randomly selecting a solution
$x_R$ from $\{x_1, \ldots, x_N\}$, another possibility would be to output
the solution $\hat x_N$ such that
\beq \label{DSG}
\|\nabla f(\hat x_N)\| = \min_{k=1,\ldots, N} \|\nabla f(x_k)\|.
\eeq
We can show that $\bbe\|\nabla f(\hat x_N)\|$ goes to zero with
similar rates of convergence as in \eqnok{noncvx_smth_converg} and \eqnok{nocvx_smooth}.
However, to use this strategy would require some extra computational effort
to compute $\|\nabla f(x_k)\|$ for all $k = 1, \ldots, N$.
Since $\|\nabla f(x_k)\|$ cannot be computed exactly, to estimate them
by using Monte-carlo simulation would incur additional
approximation errors and raise some reliability issues.
%If such a strategy is going to be used,
%one could prefer to taking more
%samples at each step $k$, in order to reduce the variation $\sigma$ of the $\SFO$
%and hence the total number of steps (as well as the extra computational
%effort for computing $\nabla f(x_k)$, $k = 1, \ldots, N$) will be reduced.
On the other hand, the above RSG method does not require any extra
computational effort for estimating the gradients $\|\nabla f(x_k)\|$, $k = 1, \ldots, N$.

Secondly, observe that in the stepsize policy \eqnok{constant_step},
we need to specify a parameter $\tilde D$. While the RSG method converges
for any arbitrary $\tilde D>0$,
it can be easily seen from
\eqnok{nocvx_smooth} and \eqnok{cvx_smooth} that an optimal selection
of $\tilde D$ would be $D_f$ and $D_X$, respectively, for solving
nonconvex and convex SP problems. With such selections, the bounds in \eqnok{nocvx_smooth}
and \eqnok{cvx_smooth}, respectively, reduce to
\beq \label{nocvx_smooth1}
\frac{1}{L} \bbe[\|\nabla f(x_R)\|^2]
\le \frac{L D_f^2}{N} + \frac{2 D_f\sigma}{\sqrt{N}}.
\eeq
and
\beq \label{cvx_smooth1}
\bbe[f(x_R) - f^*] \le \frac{L D_X^2}{N} + \frac{2 D_X \sigma}{\sqrt{N}}.
\eeq
Note however, that the exact values of $D_f$ or $D_X$ are rarely known and
one often need to set $\tilde D$ to a suboptimal value, e.g., certain
upper bounds on $D_f$ or $D_X$.

Thirdly, one possible drawback for the above RSG method is that
one need to estimate $L$ to obtain an upper bound on $\gamma_k$ (see, e.g., \eqnok{constant_step}),
which will also possibly affect the selection of $P_R$ (see \eqnok{prob_fun}). 
%the stepsize policy in \eqnok{constant_step}
%is that we need to estimate $L$ to obtain an upper bound on the stepsizes $\{\gamma_k\}$. 
Note that similar requirements also exist for some deterministic first-order methods (e.g., gradient descent and
Nesterov's accelerated gradient methods). While under the
deterministic setting, one can somehow relax such requirements
by using certain line-search procedures to enhance the practical performance of these methods,
it is more difficult to devise similar line-search procedures for the stochastic setting, since
the exact values of $f(x_k)$ and $\nabla f(x_k)$ are not available.
It should be noted, however, that we do not need very accurate
estimate for $L$ in the RSG method. Indeed, it can be easily checked that
the RSG method exhibits an ${\cal O}(1/\sqrt{N})$ rate of convergence
if the stepsizes $\{\gamma_k\}$ are set to
\[
\min\left\{\frac{1}{q L}, \frac{\tilde D}{\sigma\sqrt{N}}\right\}, \ \ \ k = 1, \ldots, N
\]
%$
%\min\left\{1/(q L), \tilde D/(\sigma\sqrt{N})\right\}, \ \ \ k = 1, \ldots, N
%$
for any $q \in [1, \sqrt{N}]$. In other words, we can overestimate the value of $L$ by
a factor up to $\sqrt{N}$
and the resulting RSG method still exhibits similar rate of convergence.
A common practice in stochastic optimization is to estimate
$L$ by using the stochastic gradients computed at a small number of trial 
points (see, e.g., \cite{NJLS09-1,lns11,GhaLan12-2a,GhaLan10-1b}). 
We have adopted such a strategy in our implementation of the RSG method as described in more details
in the technical report associated with this paper~\cite{GhaLan12}.
It is also worth noting that,
although in general the selection of $P_R$ will depend on $\gamma_k$ and hence on $L$,
such a dependence is not necessary in some special cases.
In particular, if the stepsizes $\{\gamma_k\}$ are chosen according to
a constant stepsize policy (e.g., \eqnok{constant_step}), then $R$ is uniformly distributed on $\{1, \ldots, N\}$.

Fourthly, it is interesting to note that the RSG method allows us to have
a unified treatment for both nonconvex and convex SP problems in view of the
specification of $\{\gamma_k\}$ and $P_R(\cdot)$ (c.f., \eqnok{prob_fun}
and \eqnok{constant_step}). Recall that
the optimal rate of convergence for solving smooth convex SP problems is given by
\[
{\cal O} \left(\frac{L D_X^2}{N^2}+\frac{D_X \sigma}{\sqrt N} \right).
\]
This bound has been obtained by Lan~\cite{Lan10-3} based on a stochastic
counterpart of Nesterov's method \cite{Nest83-1,Nest04}.
Comparing \eqnok{cvx_smooth1} with the above bound,
the RSG method possesses a nearly optimal rate of convergence, since the second term in \eqnok{cvx_smooth1}
is unimprovable while the first term in \eqnok{cvx_smooth1} can be much improved.
%Note that Lan~\cite{Lan10-3} showed that,
%by developing a stochastic counterpart of Nesterov's method \cite{Nest83-1,Nest04},
%the true optimal rate of convergence for solving stochastic smooth convex programming
%problems is given by
%\[
%{\cal O} \left(\frac{L}{N^2}+\frac{\sigma}{\sqrt N} \right).
%\]
%Hence, the first term in \eqnok{cvx_smooth1} can be much improved
%for convex programming problems.
Moreover, as shown by Cartis et al.~\cite{CarGouToi12},
the first term in \eqnok{nocvx_smooth1} for nonconvex problems is also
unimprovable for gradient descent methods.
It should be noted, however that the analysis in \cite{CarGouToi12}
applies only for gradient descent methods and  
does not show that the ${\cal O} (1/N)$ term is tight for all first-order
methods.
%exclude the
%possibility of having stronger bounds for other first-order methods.
%unless higher-order information about $f(\cdot)$ is used (see, e.g.,
%\cite{NestPola06-1}).

Finally, observe that we can use different stepsize policy other than the constant one in \eqnok{constant_step}.
In particular, it can be shown that the RSG method with the following two stepsize policies will
exhibit similar rates of convergence as those in Corollary~\ref{noncvx_smth_cor}.
\begin{itemize}
\item {\sl Increasing stepsize policy}: %One possible choice of varying stepsizes is to select an increasing one as follows.
\[
\gamma_k = \min \left\{ \frac{1}{L}, \frac{\tilde D \sqrt{k}}{\sigma N}\right\}, k = 1, \ldots, N.
\]

\item {\sl Decreasing stepsize policy}: %Another choice of the stepsizes which increase the chance of stoping the algorithm earlier can be stated as
\[
\gamma_k = \min \left\{ \frac{1}{L}, \frac{\tilde D }{\sigma (kN)^{\frac{1}{4}}}\right\}, k = 1, \ldots, N.
\]
\end{itemize}
Intuitively speaking, one may want to choose decreasing stepsizes which, according to the definition of $P_R(\cdot)$ in \eqnok{prob_fun},
can stop the algorithm earlier. On the other hand, as the algorithm moves forward and local information about the gradient gets better,
choosing increasing stepsizes might be a better option. We expect that the practical performance of these stepsize policies will
depend on each problem instance to be solved.

%note however that the exact value of $L$ is not known for the SP problems, it commonly appears in the SA type methods. Our preliminary numerical results \cite{GhaLan12} show %that the RSG method can converge by just using a simple estimation of $L$. Moreover, while constant stepsize policy does not hurt the complexity results of the RSG method, it %may affect the performance of the algorithm computationally. More specifically, one can choose decreasing stepsizes which according to the definition of $P_R(\cdot)$ in %\eqnok{prob_fun} implies to stop the algorithm earlier.
%On the other hand, as the algorithm moves forward and local information about the gradient gets better, it might be better to choose longer stepsizes. Below we present two %examples of the aforementioned varying stepsizes which do not affect the convergence rates obtained in the above corollary.

\vgap

While Theorem~\ref{noncvx_smth_theorem}
and Corollary~\ref{noncvx_smth_cor} establish the expected convergence
performance over many runs of the RSG method, we are also interested
in the large-deviation properties for a single run of this method. In particular,
we are interested in establishing its complexity for computing
an {\sl $(\epsilon, \Lambda)$-solution}
of problem \eqnok{NLP}, i.e., a point $\bar x$ satisfying
$\prob\{\|\nabla f(\bar x)\|^2 \le \epsilon\} \ge 1 - \Lambda$
for some $\epsilon >0$ and $\Lambda \in (0,1)$.
%
%Observe that, by using the results in Theorem~\ref{noncvx_smth_theorem}
%and Corollary~\ref{noncvx_smth_cor}, we can estimate the
%number of calls to $\SFO$ in order to compute
%Indeed,
By using \eqnok{nocvx_smooth} and Markov's inequality, we have
\beq \label{nocvx_smooth_prob}
\prob\left\{
\|\nabla f(x_R)\|^2 \ge  \lambda L \cB_N \right\}
\le \frac{1}{\lambda}, \ \ \forall \lambda > 0.
\eeq
It then follows that the number of
calls to $\SFO$ performed by the RSG method for finding
an $(\epsilon, \Lambda)$-solution, after disregarding a few constant factors,
can be bounded by
\beq \label{nocvx_smooth_prob1}
{\cal O} \left\{
\frac{1}{\Lambda \epsilon} +  \frac{\sigma^2}{\Lambda^2 \epsilon^2} \right\}.
\eeq
%\beq \label{nocvx_smooth_prob1}
%{\cal O} \left\{
%\frac{L^2 D_f^2}{\Lambda \epsilon} + \frac{L^2}{\Lambda^2} \left(\tilde D + \frac{D_f^2}{\tilde D} \right)^2
%\frac{\sigma^2}{\epsilon^2} \right\}.
%\eeq
The above complexity bound is rather pessimistic in terms of its dependence on $\Lambda$.
We will investigate one possible way to significantly improve it in next subsection.

\subsection{A two-phase randomized stochastic gradient method} \label{sec_2RSG}
In this section, we describe a variant of the RSG method which can considerably improve the complexity
bound in \eqnok{nocvx_smooth_prob1}. This procedure consists of two phases: an optimization phase used to
generate a list of candidate solutions via a few independent runs of the RSG method
and a post-optimization phase in which a solution is selected
from this candidate list.

\vgap

\noindent {\bf A two-phase RSG ($2$-RSG) method}
\begin{itemize}
\item [] {\bf Input:}
Initial point $x_1$, number of runs $S$, iteration limit $N$, and sample size $T$.
%\item [] {\bf Initialize:} Set $S = \lceil \log 2/\Lambda\rceil$.
\item [] {\bf Optimization phase:}
\item [] \hspace{0.1in} For $s = 1, \ldots, S$
\begin{itemize}
	\item [] Call the RSG method with input $x_1$,
iteration limit $N$, stepsizes $\{\gamma_k\}$ in \eqnok{constant_step} and
probability mass function $P_{R}$ in \eqnok{prob_fun}. Let $\bar x_s$ be
the output of this procedure.
\end{itemize}
\item [] {\bf Post-optimization phase:}
\item[] \hspace{0.1in} Choose a solution $\bar x^*$ from the candidate
list $\{\bar x_1, \ldots, \bar x_S\}$ such that
\beq \label{post_opt}
\|g(\bar x^*)\| = \min_{s = 1, \ldots, S} \|g(\bar x_s)\|, \ \
g(\bar x_s) := \frac{1}{T} \sum_{k=1}^T G(\bar x_s, \xi_k),
\eeq
where $G(x, \xi_k)$, $k = 1, \ldots, T$, are the stochastic gradients
returned by the $\SFO$.
\end{itemize}

Observe that in \eqnok{post_opt}, we define the best solution $\bar x^*$ as the one with the smallest
value of $\|g(\bar x_s)\|$, $s = 1, \ldots, S$. Alternatively, one can choose
$\bar x^*$ from $\{\bar x_1, \ldots, \bar x_S\}$ such that
\beq \label{post_opt2}
\tilde f(\bar x^*) = \min_{1, \ldots, S} \tilde f(\bar x_s), \ \
\tilde f(\bar x_s) = \frac{1}{T} \sum_{k=1}^T F(\bar x_s, \xi_k).
\eeq
It should be noted that the $2$-RSG method is different from
a two-phase procedure for convex stochastic programming by Nesterov and Vial~\cite{NesVia00},
where the average of $\bar x_1, \ldots, \bar x_S$ is chosen as the output solution.

In the $2$-RSG method described above, the number of calls
to the $\SFO$ are given by $S \times N$ and $S \times T$,
respectively, for the optimization phase and post-optimization phase.
Also note that we can possibly recycle the same sequence $\{\xi_k\}$ across
all gradient estimations in the post-optimization phase of $2$-RSG method.
We will provide in Theorem~\ref{2RSG_theorem} below certain bounds on $S$, $N$ and $T$,
to compute an $(\epsilon, \Lambda)$-solution of problem \eqnok{NLP}.

We need the following results regarding the large deviations of
vector valued martingales (see, e.g., Theorem 2.1 of \cite{jn08-1}).

\begin{lemma} \label{marting}
Assume that we are given a polish space with Borel probability measure $\mu$ and
a sequence of ${\cal F}_0 = \{\emptyset,\Omega\} \subseteq {\cal F}_1
\subseteq {\cal F}_2 \subseteq \ldots$ of $\sigma$-sub-algebras of Borel
$\sigma$-algebra of $\Omega$.
Let $\zeta_i \in \bbr^n$, $i = 1, \ldots, \infty$, be a martingale-difference sequence
of Borel functions on $\Omega$ such that $\zeta_i$ is ${\cal F}_i$  measurable and
$\bbe[\zeta_i| i-1] =0$, where $\bbe[\cdot|i]$, $i = 1, 2, \ldots$, denotes the
conditional expectation w.r.t. ${\cal F}_i$ and $\bbe \equiv \bbe[\cdot|0]$ is the expectation w.r.t.
$\mu$.
\begin{itemize}
\item [a)] If $\bbe[\|\zeta_i\|^2] \le \sigma_i^2$ for any $i \ge 1$, then
$
\bbe[\|\sum_{i=1}^N \zeta_i\|^2] \le \sum_{i=1}^N \sigma_i^2$. As a consequence, we have
\[
\forall N \ge 1, \lambda \ge 0: \prob\left\{
\|\sum_{i=1}^N \zeta_i \|^2 \ge \lambda \sum_{i=1}^N \sigma_i^2
\right\} \le \frac{1}{\lambda};
\]
\item [b)] If $\bbe\left\{\exp\left(
\|\zeta_i\|^2 /\sigma_i^{2}  \right)|i-1 \right\} \le \exp(1)$ almost surely for any $i \ge 1$, then
\[
\forall N \ge 1, \lambda \ge 0: \prob\left\{
\|\sum_{i=1}^N \zeta_i \| \ge \sqrt{2} (1+ \lambda) \sqrt{\sum_{i=1}^N \sigma_i^2}
\right\} \le \exp(-\lambda^2 /3).
\]
\end{itemize}
\end{lemma}

\vgap

We are now ready to describe the main convergence properties of the $2$-RSG method.
More specifically, Theorem~\ref{2RSG_theorem}.a) below shows the convergence
rate of this algorithm for a given set of parameters $(S,N,T)$, while Theorem~\ref{2RSG_theorem}.b)
establishes the complexity of the $2$-RSG method for computing an
$(\epsilon, \Lambda)$-solution of problem \eqnok{NLP}.

\begin{theorem} \label{2RSG_theorem}
Under Assumption A1, the following statements hold for the $2$-RSG method applied
to problem \eqnok{NLP}.
\begin{itemize}
\item [a)] Let ${\cal B}_N$ be defined in \eqnok{nocvx_smooth}. We have
\beq \label{2RSG_conv1}
\prob\left\{
\|\nabla f(\bar x^*)\|^2 \ge 2 \left(4 L {\cal B}_N + \frac{3 \lambda \sigma^2}{T} \right)
\right\} \le \frac{S+1}{\lambda} + 2^{-S}, \ \
\forall \, \lambda > 0;
\eeq
\item [b)] Let $\epsilon > 0$ and $\Lambda \in (0,1)$ be given. If the parameters $(S,N,T)$ are set to
\beqa
S &=& S(\Lambda) := \left \lceil \log (2/ \Lambda) \right \rceil, \label{def_S}\\
N &=& N(\epsilon) := \left \lceil \max \left\{ \frac{32 L^2 D_f^2}{\epsilon}, \left[32 L \left( \tilde D + \frac{D_f^2}{\tilde D}\right)
\frac{\sigma}{\epsilon} \right]^2 \right\} \right \rceil, \label{def_N} \\
T &=& T(\epsilon,\Lambda):= \left \lceil \frac{24 (S+1) \sigma^2}{\Lambda \epsilon} \right \rceil,\label{def_Tbar}
\eeqa
then the $2$-RSG method can compute an $(\epsilon, \Lambda)$-solution of problem \eqnok{NLP}
after taking at most
\beq \label{bnd_compl}
S(\Lambda) \, \left[ N(\epsilon) + T(\epsilon, \Lambda)\right]
\eeq
calls to the stochastic first-order oracle.
\end{itemize}
\end{theorem}

\begin{proof}
We first show part a).
Observe that by the definition of $\bar x^*$ in \eqnok{post_opt}, we have
\beqas
\|g(\bar x^*)\|^2 &=& \min_{s=1,\ldots,S} \|g(\bar x_s)\|^2
= \min_{s=1,\ldots,S} \|\nabla f(\bar x_s) + g(\bar x_s) - \nabla f(\bar x_s)\|^2 \\
&\le& \min_{s=1, \ldots,S} \left\{2 \|\nabla f(\bar x_s)\|^2 + 2
\|g(\bar x_s) - \nabla f(\bar x_s) \|^2 \right\} \\
&\le& 2 \min_{s=1, \ldots,S} \|\nabla f(\bar x_s)\|^2 +
2 \max_{s=1,\ldots,S} \|g(\bar x_s) - \nabla f(\bar x_s) \|^2,
\eeqas
which implies that
%\beq\label{rel_post}
%\begin{array}{l}
\beqa
\|\nabla f(\bar x^*)\|^2 &\le& 2 \|g(\bar x^*)\|^2 + 2 \|\nabla f(\bar x^*) - g(\bar x^*)\|^2
\le 4
\min_{s=1, \ldots,S} \|\nabla f(\bar x_s)\|^2  \nn \\
&& \,\, + \,
4 \max_{s=1,\ldots,S} \|g(\bar x_s) - \nabla f(\bar x_s) \|^2 +
 2 \|\nabla f(\bar x^*) - g(\bar x^*)\|^2 . \label{rel_post}
% \end{array}
\eeqa
We now provide certain probabilistic upper bounds to the three terms in the
right hand side of the above inequality. Firstly, using the fact
that $\bar x_s$, $1\le s \le S$, are independent and
relation \eqnok{nocvx_smooth_prob} (with $\lambda = 2$), we have
\beq \label{opt_phase_result}
\prob \left\{
\min_{s= 1, \ldots, S} \|\nabla f(\bar x_s)\|^2 \ge
2 L \cB_N \right\} = \prod_{s=1}^S \prob\left\{\|\nabla f(\bar x_s)\|^2 \ge
2 L \cB_N \right\} \le 2^{-S}.
% = \frac{\Lambda}{2}.
\eeq
Moreover, denoting $\delta_{s,k} = G(\bar x_s, \xi_k) - \nabla f(\bar x_s)$, $k = 1, \ldots, T$,
we have $g(\bar x_s) - \nabla f(\bar x_s) = \sum_{k=1}^T \delta_{s,k} / T$. Using this observation,
Assumption A1 and Lemma~\ref{marting}.a), we conclude that, for any $s = 1, \ldots, S$,
\[
\prob\left\{ \|g(\bar x_s) - \nabla f(\bar x_s)\|^2 \ge  \frac{\lambda \sigma^2}{T} \right\}
= \prob\left\{ \|\sum_{k=1}^T \delta_{s,k} \|^2 \ge  \lambda T \sigma^2 \right\}
\le \frac{1}{\lambda}, \ \forall \lambda > 0,
\]
which implies that
\beq \label{closeness1}
\prob\left\{
\max_{s=1,\ldots,S} \|g(\bar x_s) - \nabla f(\bar x_s)\|^2 \ge \frac{\lambda \sigma^2}{T}
 \right\} \le \frac{S}{\lambda},  \ \forall \lambda > 0,
\eeq
and that
\beq \label{closeness2}
\prob\left\{
\|g(\bar x^*) - \nabla f(\bar x^*)\|^2  \ge  \frac{\lambda \sigma^2}{T}
\right\} \le \frac{1}{\lambda}, \ \forall \lambda > 0.
\eeq
The result then follows by combining relations \eqnok{rel_post}, \eqnok{opt_phase_result},
\eqnok{closeness1} and \eqnok{closeness2}.

We now show that part b) holds.
Since the $2$-RSG method needs to call the RSG method $S$ times with iteration limit $N(\epsilon)$
in the optimization phase, and estimate the gradients $g(\bar x_s)$, $s = 1, \ldots, S$ with sample
size $T(\epsilon)$ in the post-optimization phase, the total number of calls to the stochastic first-order oracle is bounded
by $S [N(\epsilon) + T(\epsilon)]$. It remains to show that $\bar x^*$ is an $(\epsilon, \Lambda)$-solution
of problem \eqnok{NLP}. Noting that by the definitions of $\cB_N$ and $N(\epsilon)$, respectively,
in \eqnok{nocvx_smooth} and \eqnok{def_N}, we have
\[
\cB_{N(\epsilon)} = \frac{L D_f^2}{N(\epsilon)} +
\left( \tilde D + \frac{D_f^2}{\tilde D}\right) \frac{\sigma}{\sqrt{N(\epsilon)}}
\le \frac{\epsilon}{32L} + \frac{\epsilon}{32L} = \frac{\epsilon}{16L}.
\]
Using the above observation, \eqnok{def_Tbar} and setting $\lambda = [2 (S+1)]/\Lambda$ in \eqnok{2RSG_conv1},
we have
\[
4 L B_{N(\epsilon)} + \frac{3 \lambda \sigma^2}{T(\epsilon)}
= \frac{\epsilon}{4} + \frac{\lambda \Lambda \epsilon}{8 (S+1)} = \frac{\epsilon}{2},
\]
which, together with relations  \eqnok{2RSG_conv1} and \eqnok{def_S},
and the selection of $\lambda$, then imply that
\[
\prob\left\{\|\nabla f(\bar x^*)\|^2  \ge \epsilon \right\} \le \frac{\Lambda}{2} +
2^{-S} \le \Lambda.
\]
\end{proof}

\vgap

It is interesting to compare the complexity bound in \eqnok{bnd_compl} with the one in
\eqnok{nocvx_smooth_prob1}. In view of \eqnok{def_S}, \eqnok{def_N} and \eqnok{def_Tbar},
the complexity bound in \eqnok{bnd_compl},
after disregarding a few constant factors, is equivalent to
\beq \label{improved_compl}
{\cal O} \left\{
\frac{\log (1/\Lambda) }{\epsilon}  +
\frac{\sigma^2}{\epsilon^2} \log \frac{1}{\Lambda}
+ \frac{\log^2(1/\Lambda) \sigma^2 }{\Lambda \epsilon}
\right\}.
\eeq
%\beq \label{improved_compl}
%{\cal O} \left\{
%\frac{\log (1/\Lambda) L^2 D_f^2}{\epsilon}  + L^2 \left(\tilde D + \frac{D_f^2}{\tilde D}\right)^2
%\frac{\sigma^2}{\epsilon^2} \log \frac{1}{\Lambda}
%+ \frac{\log^2(1/\Lambda) \sigma^2 }{\Lambda \epsilon}
%\right\}.
%\eeq
The above bound can be considerably smaller than the one in \eqnok{nocvx_smooth_prob1}
up to a factor of
$
1 / \left[\Lambda^2 \log(1/\Lambda)  \right],
$
when the second terms are the dominating ones in both bounds.

\vgap

The following result shows that the bound \eqnok{bnd_compl} obtained in Theorem~\ref{2RSG_theorem} can
be further improved under certain light-tail assumption of $\SFO$.

\begin{corollary}
Under Assumptions A1 and A2, the following statements hold for the $2$-RSG method
applied to problem \eqnok{NLP}.
\begin{itemize}
\item [a)] Let ${\cal B}_N$ is defined in \eqnok{nocvx_smooth}. We have, $\forall \, \lambda > 0$,
\beq \label{2RSG_conv1_1}
\prob\left\{
\|\nabla f(\bar x^*)\|^2 \ge 4 \left[ 2 L \cB_N + 3 (1 + \lambda)^2 \frac{\sigma^2}{T} \right]
\right\} \le (S+1) \exp(-\lambda^2/3) + 2^{-S};
\eeq
\item [b)] Let $\epsilon > 0$ and $\Lambda \in (0,1)$ be given.
If $S$ and $N$ are set to $S(\Lambda)$ and $N(\epsilon)$ as in
\eqnok{def_S} and \eqnok{def_N}, respectively, and
the sample size $T$ is set to
\beq \label{def_Tbar1}
T = T'(\epsilon, \Lambda):= \frac{24 \sigma^2}{\epsilon}\left[1 + \left(3 \ln \frac{2(S+1)}{\Lambda} \right)^\frac{1}{2} \right]^2,
\eeq
then the $2$-RSG method can compute an $(\epsilon, \Lambda)$-solution of problem \eqnok{NLP}
in at most
\beq \label{bnd_comp2}
S(\Lambda) \left[ N(\epsilon) + T'(\epsilon, \Lambda)\right]
\eeq
calls to the stochastic first-order oracle.
\end{itemize}
\end{corollary}

\begin{proof}
We provide the proof of part a) only, since part b) follows immediately from part a) and
an argument similar to the one used in the proof of Theorem~\ref{2RSG_theorem}.b).
Denoting $\delta_{s,k} = G(\bar x_s, \xi_k) - \nabla f(\bar x_s)$, $k = 1, \ldots, T$,
we have $g(\bar x_s) - \nabla f(\bar x_s) = \sum_{k=1}^T \delta_{s,k} / T$. Using this observation,
Assumption A2 and Lemma~\ref{marting}.b), we conclude that, for any $s = 1, \ldots, S$ and $\lambda > 0$,
\[
\begin{array}{ccc}
\prob\left\{ \|g(\bar x_s) - \nabla f(\bar x_s)\|^2 \ge  2(1+\lambda)^2 \frac{\sigma^2}{T} \right\} &&\\
= \prob\left\{ \|\sum_{k=1}^T \delta_{s,k} \| \ge  \sqrt{2T} (1+\lambda) \sigma \right\}
\le \exp(-\lambda^2/3),
\end{array}
\]
which implies that
\beq \label{closeness1_p}
%\begin{array}{ccc}
\prob\left\{
\max_{s=1,\ldots,S} \|g(\bar x_s) - \nabla f(\bar x_s)\|^2 \ge 2(1+\lambda)^2 \frac{\sigma^2}{T}
 \right\} \le S \exp(-\lambda^2/3),  \ \forall \lambda > 0.
% \end{array}
\eeq
and that
\beq \label{closeness2_p}
\prob\left\{
\|g(\bar x^*) - \nabla f(\bar x^*)\|^2  \ge  2(1+\lambda)^2 \frac{\sigma^2}{T}
\right\} \le \exp(-\lambda^2/3), \ \forall \lambda > 0.
\eeq
The result in part a) then follows by combining relations \eqnok{rel_post}, \eqnok{opt_phase_result},
\eqnok{closeness1_p} and \eqnok{closeness2_p}.
\end{proof}

\vgap

In view of \eqnok{def_S}, \eqnok{def_N} and \eqnok{def_Tbar1}, the bound in \eqnok{bnd_comp2},
after disregarding a few constant factors, is equivalent to
\beq \label{improved_compl1}
{\cal O} \left\{
\frac{\log (1/\Lambda) }{\epsilon} +
\frac{\sigma^2}{\epsilon^2}  \log \frac{1}{\Lambda}
+ \frac{\log^2(1/\Lambda) \sigma^2 }{\epsilon}
\right\}.
\eeq
Clearly, the third term of the above bound is significantly
smaller than the corresponding one in \eqnok{improved_compl}
by a factor of $1/\Lambda$.

\setcounter{equation}{0}
\section{Stochastic zeroth-order methods} \label{sec_zero}
Our problem of interest in this section is problem \eqnok{NLP}
with $f$ given in \eqnok{case1}, i.e.,
\beq \label{NLP1}
f^* := \inf_{x \in \bbr^n} \left\{ f(x) := \int_{\Xi} F(x,\xi) dP(\xi) \right\}.
\eeq
Moreover, we
assume that $F(x, \xi) \in {\cal C}_L^{1,1}(\bbr^n)$ almost surely,
which clearly implies $f(x) \in {\cal C}_L^{1,1}(\bbr^n)$.
Our goal in this section is to specialize the RSG and 2-RSG method,
respectively, in Subsections~\ref{sec_RSGF} and \ref{sec_2RSGF},
to deal with the situation when only stochastic zeroth-order information of $f$ is available.

\subsection{The randomized stochastic gradient free method} \label{sec_RSGF}
Throughout this section, we assume that $f$ is represented by
 a {\sl stochastic zeroth-order oracle ($\cal SZO$)}.
More specifically, at the $k$-th iteration,  $x_k$ and $\xi_k$ being the input,
the $\cal SZO$ outputs the quantity $F(x_k, \xi_k)$ such that
the following assumption holds:

\vgap

{\bf A3:} For any $k \ge 1$, we have
\beqa
\, \, \bbe[F(x_k, \xi_k)] = f(x_k). \label{ass3.a}
\eeqa

To exploit zeroth-order information, we consider a smooth approximation of the objective function $f$.
It is well-known (see, e.g., \cite{RocWet98}, \cite{DuBaMaWa11} and \cite{YoNeSh12}) that the convolution of
$f$ with any nonnegative, measurable and bounded function $\psi: \bbr^n  \to \bbr$ satisfying $\int_{\bbr^n} \psi(u) \,du = 1$ is an
approximation of $f$ which is at least as smooth as $f$. One of the most important examples of the function
$\psi$ is the probability density function. Here, we use the Gaussian distribution in
the convolution. Let $u$ be $n$-dimensional standard Gaussian random vector and $\mu>0$ be the smoothing parameter.
Then, a smooth approximation of $f$ is defined  as
\beq \label{rand_smooth_func}
f_{\mu}(x) = \frac{1}{(2 \pi)^{\frac{n}{2}}} \int f(x+\mu u) e^{-\frac{1}{2}\|u\|^2} \,du =\bbe_u[f(x+\mu u)].
\eeq
The following result due to Nesterov \cite{Nest11-1} describes some properties of $f_{\mu}(\cdot)$.

\begin{theorem} \label{smth_approx}
The following statements hold for any $f \in {\cal C}^{1,1}_L$.
\begin{itemize}
\item [a)] The gradient of $f_{\mu}$ given by
\beq \label{smth_approx_grad}
\nabla f_{\mu}(x) = \frac{1}{(2 \pi)^{\frac{n}{2}}} \int \frac{f(x+\mu u)-f(x)}{\mu} u e^{-\frac{1}{2}\|u\|^2} \,du,
\eeq
is Lipschitz continuous with constant $L_{\mu}$ such that $L_{\mu} \le L$;
\item [b)] For any $x \in \bbr^n$,
\beqa
|f_{\mu}(x)-f(x)| &\le& \frac{\mu^2}{2} L n, \label{rand_smth_close}\\
\|\nabla f_{\mu}(x) - \nabla f(x)\| &\le& \frac{\mu}{2}L (n+3)^{\frac{3}{2}}; \label{grad_smth_close}
\eeqa
\item [c)]
%If $p \ge 2$, then
%\beq \label{exp_norm}
%\bbe_u[\|u\|^p] \le (p+n)^{p/2}.
%\eeq
%Moreover, we have,
For any $x \in \bbr^n$,
\beq \label{stoch_smth_approx_grad}
\frac{1}{\mu^2}\bbe_u[\{f(x+\mu u)-f(x)\}^2\|u\|^2] \le \frac{ \mu^2}{2}L^2(n+6)^3 + 2(n+4)\|\nabla f(x)\|^2.
\eeq
\end{itemize}

\end{theorem}

It immediately follows from \eqnok{grad_smth_close} that
\beqa
\|\nabla f_{\mu}(x)\|^2 &\le& 2 \|\nabla f(x)\|^2+ \frac{ \mu^2}{2} L^2 (n+3)^3, \label{grad_smth_close2}\\
\|\nabla f(x)\|^2 &\le& 2\|\nabla f_{\mu}(x)\|^2+ \frac{\mu^2}{2} L^2 (n+3)^3.
\eeqa
Moreover, denoting
\beq \label{def_f_mu_s}
f_\mu^* := \min_{x \in \bbr^n} f_\mu(x),
\eeq
we conclude from \eqnok{rand_smth_close} that $|f_\mu^* - f^*| \le \mu^2 L n / 2$ and hence
that
\beq \label{closef}
 - \mu^2 L n \le [f_\mu(x) - f_\mu^*] - [f(x) - f^*] \le \mu^2 L n.
\eeq
\vgap

Below we modify the RSG method in subsection \eqnok{sec_RSG}
to use stochastic zeroth-order rather than first-order information for solving problem \eqnok{NLP1}.

\vgap

\noindent {\bf A randomized stochastic gradient free (RSGF) method}
\begin{itemize}
\item [] {\bf Input:} Initial point $x_1$,
iteration limit $N$, stepsizes $\{\gamma_k\}_{k \ge 1}$, probability mass function $P_{R}(\cdot)$
supported on $\{1, \ldots, N\}$.
\item[] {\bf Step} $0$. Let $R$ be a random variable with probability mass function $P_{R}$.
\item [] {\bf Step} $k = 1, \ldots, R$. Generate $u_k$ by Gaussian random vector generator and
call the stochastic zeroth-order oracle for computing $G_{\mu}(x_k, \xi_k, u_k)$ given by
\beq \label{grad_free_def}
G_{\mu}(x_k, \xi_k, u_k)=\frac{F(x_k+\mu u_k, \xi_k)-F(x_k,\xi_k)}{\mu}u_k.
\eeq
Set
\beq \label{update_SGF}
x_{k+1} = x_k - \gamma_k G_{\mu}(x_k, \xi_k, u_k).
\eeq
\item[]  {\bf Output} $x_R$.
\end{itemize}

\vgap

Note that the esimator $G_{\mu}(x_k, \xi_k, u_k)$ of $\nabla f_\mu(x_k)$ in \eqnok{grad_free_def}
was suggested by Nesterov in \cite{Nest11-1}. Indeed, by \eqnok{smth_approx_grad} and Assumption A3, we have

\beq \label{def_ex_G_mu}
\bbe_{\xi, u}[G_{\mu}(x, \xi, u)] = \bbe_u \left[ \bbe_{\xi}[G_{\mu}(x, \xi, u)|u] \right] = \nabla f_{\mu}(x),
\eeq
which implies that $G_{\mu}(x, \xi, u)$ is an unbiased estimator of $\nabla f_{\mu}(x)$.
Hence, if the variance $\tilde \sigma^2 \equiv \bbe_{\xi, u}[\|G_{\mu}(x, \xi, u)-\nabla f_{\mu}(x)\|^2]$
is bounded, we can directly apply the convergence results in Theorem~\ref{noncvx_smth_theorem}
to the above RSGF method. However, there still exist a few problems in this approach. Firstly, we
do not know an explicit expression of the bound $\tilde \sigma^2$.
%do not know exactly
%how the variance $\tilde \sigma^2$ is bounded and related to $\sigma$.
Secondly, this approach does not provide any information
regarding how to appropriately specify the smoothing parameter $\mu$.
The latter issue is critical for the implementation of the
RSGF method.

By applying the approximation results in Theorem \ref{smth_approx} to
the functions $F(\cdot, \xi_k)$, $k = 1, \ldots, N$, and using a slightly
different convergence analysis than the one in Theorem~\ref{noncvx_smth_theorem},
we are able to obtain much refined convergence results for the above RSGF method.
%than those obtained by a direct application of Theorem~\ref{noncvx_smth_theorem}.
%Next theorem states the main convergence
%properties of RSGF method.

\begin{theorem} \label{zero-noncvx_smth_theorem}
%Let $\{x_k\}_{k \ge 0}$ be the sequence computed according to
%\eqnok{update_SA} by the RSG method.
Suppose that the stepsizes $\{\gamma_k\}$ and
the probability mass function $P_R(\cdot)$
in the RSGF method are chosen such that $\gamma_k < 1/[2(n+4)L]$ and
\beq \label{prob_fun2}
P_R(k) := \prob\{R=k\} = \frac{\gamma_k - 2 L(n+4) \gamma_k^2}
{{\sum_{k=1}^N \left[\gamma_k - 2 L(n+4) \gamma_k^2\right]} }, \ \ k= 1,...,N.
\eeq
Then, under Assumptions A1 and A3,
\begin{itemize}
\item [a)] for any $N \ge 1$, we have
\beq \label{zero_noncvx_smth_converg}
\begin{array}{l}
\frac{1}{L} \bbe[\|\nabla f(x_R)\|^2]
\le \frac{1}{\sum_{k=1}^N \left[\gamma_k - 2 L(n+4) \gamma_k^2\right]} \left [ D_f^2 + 2 \mu^2 (n+4) \right.\\
\ \ \ \ \left. \left(1+ L(n+4)^2 \sum_{k=1}^N (\frac{\gamma_k}{4}+L \gamma_k^2)\right)
+ 2(n+4)\sigma^2 \sum_{k=1}^N \gamma_k^2 \right],
%& \frac{D_f^2 + 2 \mu^2 (n+4) \left[1+ L(n+4)^2 \sum_{k=1}^N (\frac{\gamma_k}{4}+L \gamma_k^2)\right]
%+ 2(n+4)\sigma^2 \sum_{k=1}^N \gamma_k^2}{\sum_{k=1}^N \left[\gamma_k - 2 L(n+4) \gamma_k^2\right]},
\end{array}
\eeq
where the expectation is taken with respect to $R$, $\xi_{[N]}$ and $u_{[N]}$,
and $D_f$ is defined in\eqnok{def_Df};

\item [b)] if, in addition, problem \eqnok{NLP1} is convex with an optimal solution $x^*$, then, for any $N \ge 1$,
\beq \label{zero_cvx_smth_converg}
\begin{array}{l}
\bbe[f(x_R) - f^*] \le \frac{1}{2\sum_{k=1}^N \left[\gamma_k - 2 (n+4)L \gamma_k^2\right]}
\left[ D_X^2+ 2\mu^2 L (n+4) \right.\\
\ \ \ \ \left. \sum_{k=1}^N \left[\gamma_k + L (n+4)^2 \gamma_k^2 \right]+ 2(n+4) \sigma^2 \sum_{k=1}^N \gamma_k^2 \right],
% \frac{D_X^2+ 2\mu^2 L (n+4)  \sum_{k=1}^N \left[\gamma_k + L (n+4)^2 \gamma_k^2 \right]+ 2(n+4) \sigma^2 \sum_{k=1}^N \gamma_k^2} {2\sum_{k=1}^N \left[\gamma_k - 2 (n+4)L %\gamma_k^2\right]},
\end{array}
\eeq
where the expectation is taken with respect to $R$, $\xi_{[N]}$ and $u_{[N]}$, and
$D_X$ is defined in \eqnok{def_Dx}.

\end{itemize}

\end{theorem}

\vgap

\begin{proof}
Let $\zeta_k \equiv (\xi_k, u_k)$, $k \ge 1$, $\zeta_{[N]}:=(\zeta_1,...,\zeta_{N})$,
and $\bbe_{\zeta[N]}$ denote the expectation w.r.t. $\zeta_{[N]}$.
Also denote $\Delta_k \equiv G_{\mu}(x_k, \xi_k, u_k) - \nabla f_{\mu}(x_k) \equiv
 G_{\mu}(x_k,\zeta_k) - \nabla f_{\mu}(x_k),\ \  k \ge 1$.
Using the fact that $f \in {\cal C}_L^{1,1}(\bbr^n)$, Theorem~\ref{smth_approx}.a),
\eqnok{smooth} and \eqnok{update_SGF}, we have, for any $k = 1, \ldots, N$,
\beq \label{zero_noncvx_smth_cnvg1}
\begin{array}{lll}
f_{\mu}(x_{k+1}) &\le& f_{\mu}(x_k)-\gamma_k \, \langle \nabla f_{\mu}(x_k), G_{\mu}(x_k, \zeta_k) \rangle
+ \frac{L}{2} \gamma_k^2 \, \|G_{\mu}(x_k, \zeta_k)\|^2 \\
&=& f_{\mu}(x_k) -\gamma_k \, \|\nabla f_{\mu}(x_k)\|^2 - \gamma_k \, \langle \nabla f_{\mu}(x_k), \Delta_k \rangle
+ \frac{L}{2} \gamma_k^2 \, \|G_{\mu}(x_k, \zeta_k)\|^2.
\end{array}
\eeq
Summing up these inequalities, re-arranging the terms and noting that $f^*_{\mu} \le f_{\mu}(x_{N+1})$, we obtain
\beq \label{zero_noncvx_smth_cnvg2}
\sum_{k=1}^N \gamma_k \, \|\nabla f_{\mu}(x_k)\|^2
\le f_{\mu}(x_1) - f^*_{\mu} - \sum_{ k=1}^N \gamma_k \, \langle \nabla f_{\mu}(x_k), \Delta_k \rangle
+ \frac{L}{2}\sum_{k=1}^N \gamma_k^2 \, \|G_{\mu}(x_k, \zeta_k)\|^2.
\eeq
Now, observe that by \eqnok{def_ex_G_mu},
\beq \label{ass3_seq1}
\bbe[\langle \nabla f_{\mu}(x_k), \Delta_k \rangle |\zeta_{[k-1]}] = 0.
\eeq
and that by the assumption $F(\cdot, \xi_k) \in {\cal C}_L^{1,1}(\bbr^n)$,
\eqnok{stoch_smth_approx_grad} (with $f= F(\cdot, \xi_k)$), and \eqnok{grad_free_def},
\beq \label{exp_Gmu_bnd}
\begin{array}{lll}
\bbe[\|G_{\mu}(x_k, \zeta_k)\|^2|\zeta_{[k-1]}] &\le&
2(n+4)\bbe[\|G(x_k, \xi_k)\|^2|\zeta_{[k-1]}] + \frac{\mu^2}{2} L^2 (n+6)^3  \\
&\le& 2(n+4) \left[\bbe[\|\nabla f(x_k)\|^2|\zeta_{[k-1]}] + \sigma^2\right]
+ \frac{\mu^2}{2} L^2 (n+6)^3, 
\end{array}
\eeq
where the second inequality follows from Assumption A1.
Taking expectations with respect to $\zeta_{[N]}$
on both sides of \eqnok{zero_noncvx_smth_cnvg2}
and using the above two observations, we obtain
%\beq \label{zero_noncvx_smth_cnvg4}
\[
\begin{array}{l}
\sum\limits_{k=1}^N \gamma_k \, \bbe_{\zeta[N]}\left[\|\nabla f_{\mu}(x_k)\|^2\right]
\le f_{\mu}(x_1) - f^*_{\mu} \\
\ \ \ + \, \frac{L}{2} \sum\limits_{k=1}^N \gamma_k^2 \, \left \{2(n+4) \left [\bbe_{\zeta[N]}
[\|\nabla f(x_k)\|^2] + \sigma^2\right]+ \frac{\mu^2}{2} L^2 (n+6)^3 \right \}.
\end{array}
\]
%\eeq
The above conclusion together with \eqnok{grad_smth_close2}
and \eqnok{closef} then imply that
\beq \label{zero_noncvx_smth_cnvg5}
\begin{array}{l}
\sum_{k=1}^N \gamma_k \left[\bbe_{\zeta[N]}[\| \nabla f(x_k)\|^2] - \frac{\mu^2}{2}L^2(n+3)^3 \right]
\le 2 \left[f(x_1) - f^* \right] + 2 \mu^2 L n \\
+ 2L (n+4) \sum_{k=1}^N \gamma_k^2 \, \bbe_{\zeta[N]}[\|\nabla f(x_k)\|^2]
+ \left[2L(n+4) \sigma^2 + \frac{\mu^2}{2} L^3(n+6)^3 \right] \sum_{k=1}^N \gamma_k^2.
\end{array}
\eeq
By re-arranging the terms and simplifying the constants, we have
\beq \label{zero_noncvx_smth_cnvg6}
\begin{array}{l}
\sum_{k=1}^N \left \{\left[\gamma_k - 2 L (n+4)\gamma_k^2 \right]
\bbe_{\zeta[N]}[\|\nabla f(x_k)\|^2] \right\}
\\
\le
2 \left[f(x_1) - f^* \right] + 2 L(n+4)\sigma^2 \sum_{k=1}^N \gamma_k^2 + 2 \mu^2 L n +
\frac{\mu^2}{2} L^2 \sum_{k=1}^N \left[(n+3)^3\gamma_k + L(n+6)^3\gamma_k^2 \right] \\
\le
2 \left[f(x_1) - f^* \right] + 2 L(n+4)\sigma^2 \sum_{k=1}^N \gamma_k^2 +
2 \mu^2 L (n+4) \left[1+ L(n+4)^2 \sum_{k=1}^N (\frac{\gamma_k}{4} +L \gamma_k^2)\right].
%&&+& 2\mu^2 L (n+4) \left[1+ L(n+4)^2 \sum_{k=1}^N (\gamma_k +L \gamma_k^2) \right]. \ \nn \\
\end{array}
\eeq
Dividing both sides of the above inequality by
$\sum_{k=1}^N \left[\gamma_k - 2 L (n+4) \gamma_k^2 \right]$ and
noting that
\[
\bbe[\|\nabla f(x_R)\|^2] = \bbe_{R, \zeta[N]}[\|\nabla f(x_R)\|^2] =
\frac{\sum_{k=1}^N \left\{\left[\gamma_k - 2 L(n+4) \gamma_k^2\right] \bbe_{\zeta[N]}\|\nabla f(x_k)\|^2 \right\}}
{\sum_{k=1}^N \left[\gamma_k - 2 L(n+4) \gamma_k^2\right]},
\]
we obtain \eqnok{zero_noncvx_smth_converg}.

We now show part b). Denote $\w_k \equiv \|x_k - x^*\|$.
First observe that, for any $k = 1, \ldots, N$,
\beqas
\w_{k+1}^2 &=& \|x_k - \gamma_k G_{\mu}(x_k, \zeta_k) - x^*\|^2 \\
%= \w_k^2 - 2 \gamma_k \langle G_{\mu}(x_k, \zeta_k), x_k - x^*\rangle
%+ \gamma_k^2 \|G_{\mu}(x_k, \zeta_k)\|^2 \\
&=& \w_k^2 - 2 \gamma_k \langle \nabla f_{\mu}(x_k) + \Delta_k, x_k - x^*\rangle
+ \gamma_k^2 \|G_{\mu}(x_k, \zeta_k)\|^2.
\eeqas
and hence that
\[
\w_{N+1}^2 =
\w_1^2 - 2 \sum_{k=1}^N \gamma_k \, \langle \nabla f_{\mu}(x_k), x_k - x^*\rangle
 - 2 \sum_{k=1}^N \gamma_k \, \langle \Delta_k, x_k - x^*\rangle +   \sum_{k=1}^N \gamma_k^2 \,
 \|G_{\mu}(x_k, \zeta_k)\|^2.
\]
Taking expectation w.r.t. $\zeta_{\zeta[N]}$ on both sides of the above equality,
using relation \eqnok{exp_Gmu_bnd} and noting that by \eqnok{def_ex_G_mu},
$
\bbe[\langle \Delta_k, x_k - x^*\rangle| \zeta_{[k-1]}] = 0,
$
we obtain
\beqa
\bbe_{\zeta[N]}[\w_{N+1}^2]
 &\le& \w_1^2 - 2 \sum_{k=1}^N \gamma_k \, \bbe_{\zeta[N]}
 [\langle \nabla f_{\mu}(x_k), x_k - x^*\rangle]  + 2(n+4) \sum_{k=1}^N \gamma_k^2 \,\bbe_{\zeta[N]}[\|\nabla f(x_k)\|^2]\nn \\
&& + \, \left[2(n+4)\sigma^2 + \frac{\mu^2}{2} L^2(n+6)^3\right] \sum_{k=1}^N \gamma_k^2 \nn \\
&\le& \w_1^2 - 2 \sum_{k=1}^N \gamma_k \, \bbe_{\zeta[N]}[f_{\mu}(x_k) - f_{\mu}(x^*)] + 2(n+4)L \sum_{k=1}^N \gamma_k^2 \, \bbe_{\zeta[N]}[f(x_k) - f^*]\nn \\
&& + \, \left[2(n+4)\sigma^2 + \frac{\mu^2}{2} L^2(n+6)^3 \right] \sum_{k=1}^N \gamma_k^2 \nn \\
&\le& \w_1^2 - 2 \sum_{k=1}^N \gamma_k \, \bbe_{\zeta[N]}\left[f(x_k) - f^* - \mu^2 L n\right] +
2(n+4)L \sum_{k=1}^N \gamma_k^2 \, \bbe_{\zeta[N]}[f(x_k) - f^*]\nn \\
&& + \, \left[2(n+4)\sigma^2 + \frac{\mu^2}{2} L^2(n+6)^3 \right] \sum_{k=1}^N \gamma_k^2,\nn
%&\le& \w_1^2 - 2 \sum_{k=1}^N \gamma_k \, \bbe_{\xi_{[N]},u_{[N]}}[f_{\mu}(x_k) - f_{\mu}^*] + 4(n+4)L \sum_{k=1}^N \gamma_k^2 \, \bbe_{\xi_{[N]},u_{[N]}}[f(x_k) - f^*]\nn \\
%&& + \, \left[\frac{\mu^2}{2} L^2(n+6)^3 + 2(n+4)\sigma^2\right] \sum_{k=1}^N \gamma_k^2. \nn
\eeqa
where the second inequality follows from \eqnok{basic_convex_rel} and
the convexity of $f_\mu$, and the last inequality follows from \eqnok{rand_smth_close}.
Re-arranging the terms in the above inequality, using the facts that $\w_{N+1}^2 \ge 0$ and
$f(x_k) \ge f^*$, and simplifying the constants, we have
\beqas
\lefteqn{2 \sum_{k=1}^N \left[\gamma_k  - 2(n+4)L\gamma_k^2)\right]
\bbe_{\zeta[N]}[f(x_k) - f^*] }\\
&\le& 2 \sum_{k=1}^N \left[\gamma_k  - (n+4)L\gamma_k^2)\right]
\bbe_{\zeta[N]}[f(x_k) - f^*]\\
&\le& \w_1^2 + 2 \mu^2 L (n+4) \sum_{k=1}^N \gamma_k  +2(n+4)
\left[L^2 \mu^2 (n+4)^2 + \sigma^2\right]\sum_{k=1}^N \gamma_k^2.
\eeqas
The rest of proof is similar to part a) and hence the details are skipped.
\end{proof}

\vgap

Similarly to the RSG method, we can specialize the convergence results in
Theorem~\ref{zero-noncvx_smth_theorem} for the RSGF method with
a constant stepsize policy.

\begin{corollary}
Suppose that the stepsizes $\{\gamma_k\}$ are set to
\beq \label{zero_constant_step}
\gamma_k = \frac{1}{\sqrt{n+4}} \min \left\{ \frac{1}{4 L \sqrt{n+4}}, \frac{\tilde D}{\sigma \sqrt N}\right\}, \ \ k = 1, \ldots, N,
\eeq
for some $\tilde D >0$. Also assume that the probability mass function $P_R(\cdot)$
is set to \eqnok{prob_fun2} and $\mu$ is chosen such that
\beq \label{def_mu}
\mu \le \frac{D_f}{(n+4) \sqrt{2N}}
%\mu \le \frac{1}{\sqrt {(n+4)N}} \min\left\{\frac{D_f}{\sqrt{2(n+4)}}, D_X \right\},
\eeq
where $D_f$ and $D_X$ are defined in \eqnok{def_Df} and \eqnok{def_Dx}, respectively.
Then, under Assumptions A1 and A3, we have
\beq \label{zero_nocvx_smooth}
\frac{1}{L} \bbe[\|\nabla f(x_R)\|^2] \le \bar {\cB}_N
:=\frac{12(n+4)L D_f^2}{N} + \frac{ 4 \sigma \sqrt {n+4} }{\sqrt{N}}\left(\tilde D + \frac{D_f^2}{\tilde D}\right) .
\eeq
If, in addition, problem \eqnok{NLP1} is convex with an optimal solution $x^*$
and $\mu$ is chosen such that
\[
\mu \le \frac{D_X}{\sqrt{(n+4)}},
\] then,
\beq \label{zero_cvx_smooth}
\bbe[f(x_R) - f^*] \le  \frac{5 L (n+4)D_X^2}{N} + \frac{2 \sigma \sqrt{n+4} }{\sqrt{N}}\left(\tilde D + \frac{D_X^2}{\tilde D}\right).
\eeq

\end{corollary}

\vgap

\begin{proof}
We prove \eqnok{zero_nocvx_smooth} only since relation \eqnok{zero_cvx_smooth} can be shown by using similar arguments.
First note that by \eqnok{zero_constant_step}, we have
\beqa
\gamma_k &\le& \frac{1}{4(n+4)L}, \ \ k = 1, \ldots, N, \\
\sum_{k=1}^N \left[\gamma_k - 2 L(n+4) \gamma_k^2 \right] &=& N \gamma_1 \left[1 - 2 L(n+4) \gamma_1 \right] \ge \frac {N \gamma_1}{2}.
\eeqa
Therefore, using the above inequalities and \eqnok{zero_noncvx_smth_converg}, we obtain
\beqas
\frac{1}{L} \bbe[\|\nabla f(x_R)\|^2] &\le& \frac{2 D_f^2 + 4 \mu^2(n+4)}{N \gamma_1} + \mu^2 L(n+4)^3 + 4(n+4) \left[ \mu^2 L^2(n+4)^2 + \sigma^2 \right] \gamma_1  \\
&\le& \frac{2 D_f^2+ 4 \mu^2(n+4)}{N}\max\left\{4 L(n+4), \frac{\sigma \sqrt{(n+4)N}}{\tilde D} \right\}  \\
&& + \, \mu^2 L (n+4)^2\left[(n+4) + 1 \right] + \frac {4 \sqrt {n+4}\tilde D \sigma} {\sqrt N},
\eeqas
which, in view of \eqnok{def_mu}, then implies that
\beqas
\frac{1}{L} \bbe[\|\nabla f(x_R)\|^2]&\le& \frac{2 D_f^2}{N} \left[ 1 + \frac{1}{(n+4) N}\right] \left[4 L(n+4) + \frac{\sigma \sqrt{(n+4)N}}{\tilde D} \right]\\
&& + \, \frac{L D_f^2}{2N} \left[(n+4) + 1\right] + \frac {4 \sqrt {n+4}\tilde D \sigma} {\sqrt N}\\
&=& \frac{LD_f^2}{N}\left[\frac{17(n+4)}{2} + \frac{8}{N} + \frac{1}{2} \right]
%[\frac{1}{2}+\frac{8}{N}]\frac{L D_f^2}{N} + \frac{10 L(n+4) D_f^2}{N} +
+ \frac{2 \sigma \sqrt {n+4} }{\sqrt N}
\left[ \frac{D_f^2}{\tilde D}\left(1+\frac{1}{(n+4)N} \right) + 2 \tilde D \right ]  \\
&\le& \frac{12 L(n+4) D_f^2}{N} + \frac{4 \sigma \sqrt {n+4}}{\sqrt{N}} \left(\tilde D + \frac{D_f^2}{\tilde D} \right).
\eeqas
\end{proof}

\vgap

A few remarks about the results obtained in Corollary~\ref{noncvx_smth_cor} are in order.
Firstly, similar to the RSG method, we use the same selection of stepsizes
$\{\gamma_k\}$ and probability mass function $P_R(\cdot)$ in RSGF method for both convex and
nonconvex SP problems. In particular, in view of \eqnok{zero_nocvx_smooth}, the iteration complexity of
the RSGF method for finding an $\epsilon$-solution of problem \eqnok{NLP1} can be bounded by
${\cal O}(n/\epsilon^2)$.
%, which is worse than the corresponding results
%for the RSG method up to a ${\cal O}(\sqrt{n})$ factor.
Moreover, in view of \eqnok{zero_cvx_smooth}, if the problem is convex,
a solution $\bar x$ satisfying $\bbe[f(\bar x) - f^*] \le \epsilon$ can
also be found in ${\cal O}(n/\epsilon^2)$ iterations.
This result has a weaker dependence (by a factor of $n$) than the one established
by Nesterov for solving general nonsmooth convex SP problems (see page 17 of \cite{Nest11-1}).
This improvement is obtained since we are dealing with a more special class of
SP problems.
%By explicitly taking into account the smoothness of $f$, we show that the complexity 
%established by Nesterov \cite{Nest11-1} for solving general convex SP problems
%(i.e., ${\cal O}(n^2/\epsilon^2)$) by a factor of ${\cal O}({n})$. 
Also, note that in the case of $\sigma=0$, the iteration complexity of the RSGF method reduces to ${\cal O}(n/\epsilon)$ 
which is is similar to the one obtained by Nesterov \cite{Nest11-1} for the derivative free random search 
method when applied to both smooth convex and nonconvex deterministic problems.

Secondly, we need to specify $\tilde D$ for the stepsize policy in \eqnok{zero_constant_step}.
According to \eqnok{zero_nocvx_smooth} and \eqnok{zero_cvx_smooth}, an optimal selection of $\tilde D$
would be $ D_f$ and $D_X$, respectively, for the nonconvex and convex case.
With such selections, the bounds in \eqnok{zero_nocvx_smooth}
and \eqnok{zero_cvx_smooth}, respectively, reduce to
\beqa
 \frac{1}{L} \bbe[\|\nabla f(x_R)\|^2] &\le& \frac{12(n+4)L D_f^2}{N} + \frac{ 8 \sqrt{n+4} D_f \sigma}{\sqrt N}, \label{zero_nocvx_smooth1} \\
\bbe[f(x_R) - f^*] &\le&  \frac{5 L (n+4)D_X^2}{N} + \frac{4 \sqrt{n+4} D_X \sigma}{\sqrt N}. \label{zero_cvx_smooth1}
\eeqa

\vgap

Similarly to the RSG method, we can establish the complexity of the
RSGF method for finding an $(\epsilon, \Lambda)$-solution of problem \eqnok{NLP1}
for some $\epsilon >0$ and $\Lambda \in (0,1)$.
More specifically, by using \eqnok{zero_nocvx_smooth} and Markov's inequality, we have
\beq \label{zero_nocvx_smooth_prob}
\prob\left\{
\|\nabla f(x_R)\|^2 \ge  \lambda L \bar{\cB}_N \right\}
\le \frac{1}{\lambda}, \ \ \forall \lambda > 0,
\eeq
which implies that the total  number of calls to the $\SZO$
performed by the RSGF method for finding an $(\epsilon, \Lambda)$-solution
of \eqnok{NLP1} can be bounded by
\beq \label{zero_nocvx_smooth_prob1}
{\cal O} \left\{
\frac{n L^2 D_f^2}{\Lambda \epsilon} + \frac{n L^2}{\Lambda^2} \left(\tilde D + \frac{D_f^2}{\tilde D} \right)^2
\frac{\sigma^2}{\epsilon^2} \right\}.
\eeq
We will investigate a possible approach to improve the above complexity bound in next subsection.
%It can be easily seen that the above bound linearly depends on the dimension $n$.
%Hence, decreasing the other terms will significantly improve the above bound when $n$ is big.
%Our goal in the next subsection is to improve this bound.

\vgap

\subsection{A two-phase randomized stochastic gradient free method} \label{sec_2RSGF}
In this section, we modify the 2-RSG method to improve the complexity bound in
\eqnok{zero_nocvx_smooth_prob1} for finding an $(\epsilon, \Lambda)$-solution of problem \eqnok{NLP1}.

\vgap

\noindent {\bf A two-phase RSGF ($2$-RSGF) method}
\begin{itemize}
\item [] {\bf Input:}
Initial point $x_1$, number of runs $S$, iteration limit $N$, and sample size $T$.
%\item [] {\bf Initialize:} Set $S = \lceil \log 2/\Lambda\rceil$.
\item [] {\bf Optimization phase:}
\item [] \hspace{0.1in} For $s = 1, \ldots, S$
	\begin{itemize}
	\item [] Call the RSGF method with input $x_1$,
iteration limit $N$, stepsizes $\{\gamma_k\}$ in \eqnok{zero_constant_step},
probability mass function $P_{R}$ in \eqnok{prob_fun2}, and the smoothing parameter
$\mu$ satisfying \eqnok{def_mu}. Let $\bar x_s$ be
the output of this procedure.
\end{itemize}
\item [] {\bf Post-optimization phase:}
\item[] \hspace{0.1in} Choose a solution $\bar x^*$ from the candidate
list $\{\bar x_1, \ldots, \bar x_S\}$ such that
\beq \label{zero_post_opt}
\|g_{\mu}(\bar x^*)\| = \min_{s = 1, \ldots, S} \|g_{\mu}(\bar x_s)\|, \ \
g_{\mu}(\bar x_s) := \frac{1}{T} \sum_{k=1}^T G_{\mu}(\bar x_s, \xi_k, u_k),
\eeq
where $G_\mu(x, \xi, u)$ is defined in \eqnok{grad_free_def}.
\end{itemize}

\vgap

The main convergence properties of the
$2$-RSGF method are summarized in Theorem~~\ref{2RSGF_theorem}. More specifically, Theorem~\ref{2RSGF_theorem}.a)
establishes the rate of convergence of the $2$-RSGF method with
a given set of parameters $(S,N,T)$, while
Theorem~\ref{2RSGF_theorem}.b) shows the complexity of this method
for finding an $(\epsilon,\Lambda)$-solution
of problem \eqnok{NLP1}.

\begin{theorem} \label{2RSGF_theorem}
%and also let $S \equiv \lceil \log 2/ \Lambda \rceil$.
Under Assumptions A1 and A3, the following statements hold for the $2$-RSGF method
applied to problem \eqnok{NLP1}.
\begin{itemize}
\item [a)] Let $\bar{\cB}_N$ be defined in \eqnok{zero_nocvx_smooth}. We have
\beq \label{2RSGF_conv1}
\begin{array}{l}
\prob\left\{
%\|\nabla f(\bar x^*)\|^2 \ge  9 L \bar{\cB}_N + \frac{24(n+4)\lambda}{T}
%\left(2 L \bar{\cB}_N+ \sigma^2\right)\right\}
\|\nabla f(\bar x^*)\|^2  \ge 8 L \bar{\cB}_N + \frac{3(n+4) L^2 D_f^2}{2 N}
 \,\,+ \frac{24(n+4)\lambda}{T}
\left[L \bar{\cB}_N + \frac{(n+4) L^2 D_f^2}{N} + \sigma^2\right]
\right\} \\
\le \frac{S+1}{\lambda} + 2^{-S}, \ \
\forall \, \lambda > 0;
\end{array}
\eeq
\item [b)] Let $\epsilon > 0$ and $\Lambda \in (0,1)$ be given.
If $S$ is set to $S(\Lambda)$ as in \eqnok{def_S}, and
 the iteration limit $N$ and sample size $T$, respectively, are set to
\beqa
N &=& \hat N(\epsilon) := \max \left\{ \frac{12(n+4) (6 L D_f)^2}{\epsilon},
\left[72 L \sqrt{n+4} \left(\tilde D + \frac{D_f^2}{\tilde D}\right)
\frac{\sigma}{\epsilon} \right]^2 \right\}, \label{zero_def_N} \\
T &=& \hat {T}(\epsilon,\Lambda):= \frac{24 (n+4) (S+1)}{\Lambda} \max \left\{1 ,  \frac{6 \sigma^2}{ \epsilon} \right\},
\label{def_That}
\eeqa
then the $2$-RSGF method can compute an $(\epsilon, \Lambda)$-solution of problem \eqnok{NLP1}
after taking at most
\beq \label{zero_bnd_compl}
2 S(\Lambda) \left[ \hat N(\epsilon) + \hat T(\epsilon, \Lambda)\right]
\eeq
calls to the $\SZO$.
\end{itemize}
\end{theorem}

\vgap

\begin{proof}
First, observe that by \eqnok{grad_smth_close}, \eqnok{def_mu} and \eqnok{zero_nocvx_smooth}, we have
\beq \label{CN_mu}
\|\nabla f_{\mu}(x) - \nabla f(x)\|^2 \le \frac{\mu^2}{4} L^2(n+3)^3
\le \frac{(n+4)L^2 D_f^2}{8N}.
%\le \frac{L \bar{\cB}_N}{8},
\eeq
Using this observation and the definition of $\bar x^*$ in \eqnok{zero_post_opt}, we obtain
\beqas
\|g_{\mu}(\bar x^*)\|^2 &=& \min_{s=1,\ldots,S} \|g_{\mu}(\bar x_s)\|^2
= \min_{s=1,\ldots,S} \|\nabla f(\bar x_s) + g_{\mu}(\bar x_s) - \nabla f(\bar x_s)\|^2  \\
&\le& \min_{s=1, \ldots,S} \left\{2 \left[\|\nabla f(\bar x_s)\|^2 +
\|g_{\mu}(\bar x_s) - \nabla f(\bar x_s) \|^2 \right]\right\} \\
&\le& \min_{s=1, \ldots,S} \left\{2 \left[\|\nabla f(\bar x_s)\|^2 +
2\|g_{\mu}(\bar x_s) - \nabla f_{\mu}(\bar x_s) \|^2+2\|\nabla f_{\mu}(\bar x_s) - \nabla f(\bar x_s) \|^2 \right]\right\}\\
&\le& 2 \min_{s=1, \ldots,S} \|\nabla f(\bar x_s)\|^2
+ 4 \max_{s=1,\ldots,S} \|g_{\mu}(\bar x_s) - \nabla f(\bar x_s) \|^2
+ \frac{(n+4)L^2 D_f^2}{2N},
\eeqas
which implies that
\beqa
\|\nabla f(\bar x^*)\|^2
&\le& 2\|g_{\mu}(\bar x^*)\|^2 + 2 \|\nabla f(\bar x^*) - g_{\mu}(\bar x^*)\|^2 \nn\\
&\le& 2 \|g_{\mu}(\bar x^*)\|^2 + 4 \|\nabla f_{\mu}(\bar x^*) - g_{\mu}(\bar x^*)\|^2 +4
\|\nabla f(\bar x^*) - \nabla f_{\mu}(\bar x^*)\|^2 \nn \\
&\le& 4 \min_{s=1, \ldots,S} \|\nabla f(\bar x_s)\|^2
+ 8 \max_{s=1,\ldots,S} \|g_{\mu}(\bar x_s) - \nabla f(\bar x_s) \|^2
+ \frac{(n+4)L^2 D_f^2}{N} \nn\\
&& \, +  \,\, 4 \|\nabla f_{\mu}(\bar x^*) - g_{\mu}(\bar x^*)\|^2 + 4
\|\nabla f(\bar x^*) - \nabla f_{\mu}(\bar x^*)\|^2\nn \\
&\le& 4 \min_{s=1, \ldots,S} \|\nabla f(\bar x_s)\|^2
+ 8 \max_{s=1,\ldots,S} \|g_{\mu}(\bar x_s) - \nabla f(\bar x_s) \|^2 \nn\\
&& \, + \,\, 4 \|\nabla f_{\mu}(\bar x^*) - g_{\mu}(\bar x^*)\|^2 + \frac{3(n+4)L^2 D_f^2}{2N},
%&\le& 4 \left[
%\min_{s=1, \ldots,S} \|\nabla f(\bar x_s)\|^2 + 2
%\max_{s=1,\ldots,S} \|g_{\mu}(\bar x_s) - \nabla f(\bar x_s) \|^2 +
% \|\nabla f_{\mu}(\bar x^*) - g_{\mu}(\bar x^*) \|^2\right]  \nn \\
% &+& 4 \left[\|\nabla f_{\mu}(\bar x_s) - \nabla f(\bar x_s) \|^2+ \|\nabla f(\bar x^*) - \nabla f_{\mu}(\bar x^*)\|^2\right].
\label{zero_rel_post}
\eeqa
where the last inequality also follows from \eqnok{CN_mu}.
We now provide certain probabilistic bounds on the individual terms in the right hand side of
the above inequality.
%which together with \eqnok{zero_rel_post}, clearly implies
%\beqa
%\|\nabla f(\bar x^*)\|^2 &\le& 4 \left[
%\min_{s=1, \ldots,S} \|\nabla f(\bar x_s)\|^2 + 2
%\max_{s=1,\ldots,S} \|g_{\mu}(\bar x_s) - \nabla f(\bar x_s) \|^2 +
% \|\nabla f_{\mu}(\bar x^*) - g_{\mu}(\bar x^*) \|^2\right]+L \bar{\cB}(N).  \nn \\
% \label{zero_rel_post1}
%\eeqa
Using \eqnok{zero_nocvx_smooth_prob} (with $\lambda = 2$), we obtain
\beq \label{zero_opt_phase_result}
\prob \left\{
\min_{s= 1, \ldots, S} \|\nabla f(\bar x_s)\|^2 \ge
2 L \bar{\cB}_N\right\} = \prod_{s=1}^S \prob\left\{\|\nabla f(\bar x_s)\|^2 \ge
2 L \bar{\cB}_N \right\} \le 2^{-S}. %= \frac{\Lambda}{2}.
\eeq
Moreover, denote $\Delta_{s,k} = G_{\mu}(\bar x_s, \xi_k, u_k) - \nabla f_{\mu}(\bar x_s)$, $k = 1, \ldots, T$.
Note that, similar to \eqnok{exp_Gmu_bnd}, we have
\beqas
%\bbe[\|\Delta_{s,k}\|^2] &=& \bbe [\|G_{\mu}(\bar x_s, \xi_k, u_k) - \nabla f_{\mu}(\bar x_s)\|^2] \\
%&\le& \bbe[\|G_{\mu}(\bar x_s, \xi_k, u_k)\|^2] - \bbe[\|\nabla f_{\mu}(\bar x_s)\|^2] \\
\bbe [\|G_{\mu}(\bar x_s, \xi_k, u_k)\|^2]
&\le& 2(n+4) [\bbe [\|G(\bar x_s, \xi)\|^2] + \frac{\mu^2}{2} L^2 (n+6)^3  \\
&\le& 2(n+4) [\bbe [\|\nabla f(\bar x_s)\|^2] +\sigma^2]+2 \mu^2 L^2 (n+4)^3.
\eeqas
%which together with \eqnok{exp_Gmu_bnd} and under assumption A1, implies
%\beqa
%\bbe [\|\Delta_{s,k}\|^2] &\le& 2(n+4) [\bbe [\|G(\bar x_s, \xi)\|^2] + \frac{\mu^2}{2} L^2 (n+6)^3 \nn \\
%&\le& 2(n+4) [\bbe [\|\nabla f(\bar x_s)\|^2] +\sigma^2]+2 \mu^2 L^2 (n+4)^3. \nn
%\eeqa
It then follows from the previous inequality, \eqnok{def_mu} and \eqnok{zero_nocvx_smooth}
that
%Using \eqnok{def_mu} , \eqnok{def_mu}, we obtain
\beqa
\bbe [\|\Delta_{s,k}\|^2] &=& \bbe [\|G_{\mu}(\bar x_s, \xi_k, u_k) - \nabla f_{\mu}(\bar x_s)\|^2]
\le \bbe [\|G_{\mu}(\bar x_s, \xi_k, u_k)\|^2] \nn\\
&\le&
 2(n+4) \left[L \bar{\cB}_N + \sigma^2\right] + 2 \mu^2 L^2 (n+4)^3\nn \\
&\le& 2(n+4) \left[L \bar{\cB}_N + \sigma^2+\frac{L^2D_f^2}{2N}\right]=:\cD_N. \label{def_DN}
\eeqa
Noting that $g_{\mu}(\bar x_s) - \nabla f_{\mu}(\bar x_s) = \sum_{k=1}^T \Delta_{s,k} / T$, we conclude
from \eqnok{def_DN}, Assumption A1 and Lemma~\ref{marting}.a) that, for any $s = 1, \ldots, S$,
\[
\prob\left\{ \|g_{\mu}(\bar x_s) - \nabla f_{\mu}(\bar x_s)\|^2 \ge  \frac{\lambda \cD_N}{T} \right\}
= \prob\left\{ \|\sum_{k=1}^T \Delta_{s,k} \|^2 \ge  \lambda T \cD_N \right\}
\le \frac{1}{\lambda}, \ \forall \lambda > 0,
\]
which implies that
\beq \label{zero_closeness1}
\prob\left\{
\max_{s=1,\ldots,S} \|g_{\mu}(\bar x_s) - \nabla f_{\mu}(\bar x_s)\|^2 \ge \frac{\lambda \cD_N}{T}
 \right\} \le \frac{S}{\lambda},  \ \forall \lambda > 0.
\eeq
and that
\beq \label{zero_closeness2}
\prob\left\{
\|g_{\mu}(\bar x^*) - \nabla f_{\mu}(\bar x^*)\|^2  \ge  \frac{\lambda \cD_N}{T}
\right\} \le \frac{1}{\lambda}, \ \forall \lambda > 0.
\eeq
The result then follows by combining relations \eqnok{zero_rel_post}, \eqnok{zero_opt_phase_result},\eqnok{def_DN},
\eqnok{zero_closeness1} and \eqnok{zero_closeness2}.

We now show part b) holds. Clearly, the total number of calls to $\cal SZO$ in the $2$-RSGF method
is bounded by $2 S[\hat N(\epsilon)+\hat T(\epsilon)]$. It then suffices to
show that $\bar x^*$ is an $(\epsilon, \Lambda)$-solution
of problem \eqnok{NLP1}.
Noting that by the definitions of $\bar{\cB}(N)$ and $\hat N(\epsilon)$, respectively, in \eqnok{zero_nocvx_smooth}
and \eqnok{zero_def_N}, we have
\[
\bar{\cB}_{\hat N(\epsilon)} = \frac{12(n+4)L D_f^2}{\hat N(\epsilon)} +
\frac{4 \sigma \sqrt {n+4}}{\sqrt{\hat N(\epsilon)}} \left(\tilde D + \frac{D_f^2}{\tilde D}\right)
\le \frac{\epsilon}{36L} + \frac{\epsilon}{18L} = \frac{\epsilon}{12L}.
\]
Hence, we have
\[
8L \bar B_{\hat N(\epsilon)} + \frac{3(n+4) L^2 D_f^2}{2 \hat N(\epsilon)} \le
\frac{2\epsilon}{3} + \frac{\epsilon}{288} \le \frac{17 \epsilon}{24}.
\]
Moreover, by setting $\lambda = [2 (S+1)]/\Lambda$ and using \eqnok{zero_def_N} and \eqnok{def_That}, we obtain
\beqas
\frac{24(n+4) \lambda}{T} \left[ L \bar{\cB}_{\hat N(\epsilon)} + \frac{(n+4) L^2 D_f^2}{\hat N(\epsilon)} + \sigma^2\right]
&\le& \frac{24(n+4) \lambda}{T} \left( \frac{\epsilon}{12} + \frac{\epsilon}{432} +\sigma^2\right)\\
&\le&  \frac{\epsilon}{12} + \frac{\epsilon}{432} + \frac{\epsilon}{6} \le \frac{7 \epsilon}{24}.
\eeqas
Using these two observations and relation \eqnok{2RSGF_conv1} with $\lambda = [2 (S+1)]/\Lambda$, we conclude that
\beqas
\prob\left\{\nabla f(\bar x^*)\|^2  \ge \epsilon \right\} &\le&
\prob\left\{
\|\nabla f(\bar x^*)\|^2 \ge  8 L \bar{\cB}_{\hat N(\epsilon)} + \frac{3(n+4) L^2 D_f^2}{2 \hat N(\epsilon)} \right. \\
&& \, + \left.\,\, \frac{24(n+4)\lambda}{T}
\left[L \bar{\cB}_{\hat N(\epsilon)} + \frac{(n+4) L^2 D_f^2}{\hat N(\epsilon)} + \sigma^2\right]
\right\} \\
&\le& \frac{S+1}{\lambda} + 2^{-S} = \Lambda.
\eeqas
\end{proof}

\vgap

Observe that in the view of \eqnok{def_S}, \eqnok{zero_def_N} and \eqnok{def_That},
the total number of calls to $\cal SZO$ performed by the $2$-RSGF method can be bounded by
\beq \label{zero_improved_compl}
{\cal O} \left\{
\frac{n  L^2 D_f^2
\log (1/\Lambda) }{\epsilon} + n L^2
 \left(\tilde D + \frac{D_f^2}{\tilde D}\right)^2 \frac{\sigma^2}{\epsilon^2} \log \frac{1}{\Lambda}
+ \frac{n \log^2(1/\Lambda)}{\Lambda} \left(1+ \frac{\sigma^2}{\epsilon}\right)
\right\}.
\eeq
The above bound is considerably smaller than the one in \eqnok{zero_nocvx_smooth_prob1},
up to a factor of
$
{\cal O}\left( 1/[\Lambda^2 \log(1/\Lambda)] \right),
$
when the second terms are the dominating ones in both bounds.

\section{Concluding remarks} \label{sec_concl}
In this paper, we present a class of new SA methods for solving the classical unconstrained NLP
problem with noisy first-order information. We establish a few new complexity results regarding
the computation of an $\epsilon$-solution for solving this class of problems
and show that they are nearly optimal whenever the problem is convex.
Moreover, we introduce a post-optimization phase in order to improve the large-deviation
properties of the RSG method. These
procedures, along with their complexity results, are then specialized for simulation-based optimization
problems when only stochastic zeroth-order information is available.
In addition, we show that the complexity for gradient-free methods
for smooth convex SP can have a much weaker dependence on the dimension $n$ %(by a factor of $n$) 
than that for more general nonsmooth convex SP.

%Observe that although we focus only on
%unconstrained NLP problems in this paper,
%we believe that there is potential to extend our techniques to the constrained case.
%It is also worth noting that one can possibly apply the sample average
%approximation (SAA) approach to problem \eqnok{NLP} by solving a deterministic SAA problem, namely,
%$\min_{x} \sum_{i=1}^N F(x,\xi_i)$ (e.g., \cite{sha03}).
%The SAA method is not an algorithm and convexity is not used in the estimation
%of the sample size $N$. However, such an estimation
%is based on the assumption that the SAA problems are solved (with certain precision) to optimality,
%which could be difficult to verify in nonconvex cases.
%Moreover, the estimation of
%$N$ in the SAA approach usually requires the feasible set to be bounded (see, Chapter~5 of~\cite{ShDeRu09}).
%Promising numerical results are also reported for the developed algorithms.

%It is expected that the two key ingredients of our methods, i.e., randomization and post-optimization, will be useful
%for solving other types of nonconvex SP problems. It will also be interesting to
%see whether the developed complexity results are optimal when the problem is nonconvex.

%%%%%% Bibliography %%%%%%%%%%%%%%%%%%%%%%%%%%%%%%%%%%

\bibliographystyle{plain}
\bibliography{../glan-bib}

\newcommand{\noopsort}[1]{} \newcommand{\printfirst}[2]{#1}
  \newcommand{\singleletter}[1]{#1} \newcommand{\switchargs}[2]{#2#1}
\begin{thebibliography}{10}

\bibitem{Andr98-1}
S.~Andrad\'{o}ttir.
\newblock A review of simulation optimization techniques.
\newblock {\em Proceedings of the 1998 Winter Simulation Conference}, pages
  151--158.

\bibitem{CarGouToi10-1}
C.~Cartis, N.~I.~M. Gould, and Ph.~L. Toint.
\newblock On the complexity of steepest descent, newton's and regularized
  newton's methods for nonconvex unconstrained optimization.
\newblock {\em SIAM Journal on Optimization}, 20(6):2833--2852, 2010.

\bibitem{CarGouToi12}
C.~Cartis, N.~I.~M. Gould, and Ph.~L. Toint.
\newblock On the oracle complexity of first-order and derivative-free
  algorithms for smooth nonconvex minimization.
\newblock {\em SIAM Journal on Optimization}, 22:66--86, 2012.

\bibitem{chu54}
K.L. Chung.
\newblock On a stochastic approximation method.
\newblock {\em Annals of Mathematical Statistics}, pages 463--483, 1954.

\bibitem{ConSchVic09}
A.~R. Conn, K.~Scheinberg, and L.~N. Vicente.
\newblock {\em Introduction to Derivative-Free Optimization}.
\newblock SIAM, Philadelphia, 2009.

\bibitem{DuBaMaWa11}
J.~C. Duchi, P.~L. Bartlett, and M.~J. Wainwright.
\newblock Randomized smoothing for stochastic optimization.
\newblock {\em SIAM Journal on Optimization}, 22:674--701, 2012.

\bibitem{Fu02-1}
M.~Fu.
\newblock Optimization for simulation: Theory vs. practice.
\newblock {\em INFORMS Journal on Computing}, 14:192--215, 2002.

\bibitem{Fu06a}
M.C. Fu.
\newblock Gradient estimation.
\newblock In S.~G. Henderson and B.~L. Nelson, editors, {\em Handbooks in
  Operations Research and Management Science: Simulation}, pages 575--616.
  Elsevier.

\bibitem{GarVic12}
R.~Garmanjani and L.~N. Vicente.
\newblock Smoothing and worst-case complexity for direct-search methods in
  nonsmooth optimization.
\newblock {\em IMA Journal of Numerical Analysis}, 2012.
\newblock to appear.

\bibitem{GhaLan10-1b}
S.~Ghadimi and G.~Lan.
\newblock Optimal stochastic approximation algorithms for strongly convex
  stochastic composite optimization, {II}: shrinking procedures and optimal
  algorithms.
\newblock Technical report, 2010.
\newblock {\it SIAM Journal on Optimization} (under second-round review).

\bibitem{GhaLan12-2a}
S.~Ghadimi and G.~Lan.
\newblock Optimal stochastic approximation algorithms for strongly convex
  stochastic composite optimization, {I}: a generic algorithmic framework.
\newblock {\em SIAM Journal on Optimization}, 22:1469--1492, 2012.

\bibitem{GhaLan12}
S.~Ghadimi and G.~Lan.
\newblock Stochastic first- and zeroth-order methods for nonconvex stochastic
  programming.
\newblock Technical report, Department of Industrial and Systems Engineering,
  University of Florida, Gainesville, FL 32611, USA, June 2012.
\newblock {\it SIAM Journal on Optimization} (under second-round review).

\bibitem{Glasserman91}
P.~Glasserman.
\newblock {\em Gradient Estimation via Perturbation Analysis}.
\newblock Kluwer Academic Publishers, Boston,Massachusetts, 2003.

\bibitem{JNTV05-1}
A.~Juditsky, A.~Nazin, A.~B. Tsybakov, and N.~Vayatis.
\newblock Recursive aggregation of estimators via the mirror descent algorithm
  with average.
\newblock {\em Problems of Information Transmission}, 41:n.4, 2005.

\bibitem{jn08-1}
A.~Juditsky and A.~Nemirovski.
\newblock Large deviations of vector-valued martingales in $2$-smooth normed
  spaces.
\newblock Manuscript, Georgia Institute of Technology, Atlanta, GA, 2008.
\newblock E-print: www2.isye.gatech.edu/$\sim$ nemirovs/LargeDevSubmitted.pdf.

\bibitem{JRT08-1}
A.~Juditsky, P.~Rigollet, and A.~B. Tsybakov.
\newblock Learning by mirror averaging.
\newblock {\em Annals of Statistics}, 36:2183--2206, 2008.

\bibitem{ksh}
A.~J. Kleywegt, A.~Shapiro, and T.~Homem de~Mello.
\newblock The sample average approximation method for stochastic discrete
  optimization.
\newblock {\em SIAM Journal on Optimization}, 12:479--502, 2001.

\bibitem{Lan10-3}
G.~Lan.
\newblock An optimal method for stochastic composite optimization.
\newblock {\em Mathematical Programming}, 133(1):365--397, 2012.

\bibitem{lns11}
G.~Lan, A.~Nemirovski, and A.~Shapiro.
\newblock Validation analysis of mirror descent stochastic approximation
  method.
\newblock {\em Mathematical Programming}, 134:425--458, 2012.

\bibitem{LE90-1}
P.~L\'{E}cuyer.
\newblock A unified view of the {IPA}, {SF}, and {LR} gradient estimation
  techniques.
\newblock {\em Management Science}, 36(11):1364--1383, 1990.

\bibitem{Mairal09}
J.~Mairal, F.~Bach, J.~Ponce, and G.~Sapiro.
\newblock Online dictionary learning for sparse coding.
\newblock In {\em In ICML}, pages 689--696, 2009.

\bibitem{MasBaxBarFre99}
L.~Mason, J.~Baxter, P.~Bartlett, and M.~Frean.
\newblock Boosting algorithms as gradient descent in function space.
\newblock {\em Proc. NIPS}, 12:512--518, 1999.

\bibitem{NJLS09-1}
A.~Nemirovski, A.~Juditsky, G.~Lan, and A.~Shapiro.
\newblock Robust stochastic approximation approach to stochastic programming.
\newblock {\em SIAM Journal on Optimization}, 19:1574--1609, 2009.

\bibitem{nemyud:83}
A.~Nemirovski and D.~Yudin.
\newblock {\em Problem complexity and method efficiency in optimization}.
\newblock Wiley-Interscience Series in Discrete Mathematics. John Wiley, XV,
  1983.

\bibitem{Nest83-1}
Y.~E. Nesterov.
\newblock A method for unconstrained convex minimization problem with the rate
  of convergence {$O(1/k^2)$}.
\newblock {\em Doklady AN SSSR}, 269:543--547, 1983.

\bibitem{Nest04}
Y.~E. Nesterov.
\newblock {\em Introductory Lectures on Convex Optimization: a basic course}.
\newblock Kluwer Academic Publishers, Massachusetts, 2004.

\bibitem{Nest06-2}
Y.~E. Nesterov.
\newblock Primal-dual subgradient methods for convex problems.
\newblock {\em Mathematical Programming}, 120:221--259, 2006.

\bibitem{Nest11-1}
Y.~E. Nesterov.
\newblock Random gradient-free minimization of convex functions.
\newblock Technical report, Center for Operations Research and Econometrics
  (CORE), Catholic University of Louvain, January 2010.

\bibitem{NesVia00}
Y.~E. Nesterov and J.~P. Vial.
\newblock Confidence level solutions for stochastic programming.
\newblock 2000.

\bibitem{NocWri99}
J.~Nocedal and S.~J. Wright.
\newblock {\em Numerical optimization}.
\newblock Springer-Verlag, New York, USA, 1999.

\bibitem{pol90}
B.T. Polyak.
\newblock New stochastic approximation type procedures.
\newblock {\em Automat. i Telemekh.}, 7:98--107, 1990.

\bibitem{pol92}
B.T. Polyak and A.B. Juditsky.
\newblock Acceleration of stochastic approximation by averaging.
\newblock {\em SIAM J. Control and Optimization}, 30:838--855, 1992.

\bibitem{RobMon51-1}
H.~Robbins and S.~Monro.
\newblock A stochastic approximation method.
\newblock {\em Annals of Mathematical Statistics}, 22:400--407, 1951.

\bibitem{RocWet98}
R.~T. Rockafellar and R.~J.-B. Wets.
\newblock {\em Variational analysis, ser. Grundlehren der Mathematischen
  Wissenschaften [Fundamental Principles of Mathematical Sciences]}.
\newblock Springer-Verlag, Berlin, 1998.

\bibitem{RubSha93}
R.Y. Rubinstein and A.~Shapiro.
\newblock {\em Discrete Event Systems: Sensitivity Analysis and Stochastic
  Optimization by the Score Function Method}.
\newblock John Wiley \& Sons, 1993.

\bibitem{GraSarToi08}
A.~Sartenaer S.~Gratton and Ph.~L. Toint.
\newblock Recursive trust-region methods for multiscale nonlinear optimization.
\newblock {\em SIAM Journal on Optimization}, 19:414--444, 2008.

\bibitem{sac58}
J.~Sacks.
\newblock Asymptotic distribution of stochastic approximation.
\newblock {\em Annals of Mathematical Statistics}, 29:373--409, 1958.

\bibitem{sha03}
A.~Shapiro.
\newblock Monte carlo sampling methods.
\newblock In A.~Ruszczy\'{n}ski and A.~Shapiro, editors, {\em Stochastic
  Programming}. North-Holland Publishing Company, Amsterdam, 2003.

\bibitem{Spall03}
J.C. Spall.
\newblock {\em Introduction to Stochastic Search and Optimization: Estimation,
  Simulation, and Control}.
\newblock John Wiley, Hoboken, NJ, 2003.

\bibitem{Vinc12}
L.~N. Vicente.
\newblock Worst case complexity of direct search.
\newblock {\em EURO Journal on Computational Optimization}, 2012.
\newblock to appear.

\bibitem{YoNeSh12}
F.~Yousefian, A.~Nedic, and U.~V. Shanbhag.
\newblock On stochastic gradient and subgradient methods with adaptive
  steplength sequences.
\newblock {\em Automatica}, 48:56--67, 2012.

\end{thebibliography}
\end{document}